\documentclass[11pt]{article}
\usepackage{graphics,color,amsmath,amssymb}
\usepackage[font=small,labelfont=bf]{caption}
\usepackage[pdftex]{graphicx}

\usepackage{indentfirst,amsmath,amsfonts,amssymb,amsthm,cite}
\usepackage{subeqnarray}
\usepackage{mathrsfs}
\usepackage{float,bm,booktabs,textcomp}
\usepackage{graphics,color,subfigure}
\usepackage{xcolor}

\allowdisplaybreaks  

\numberwithin{equation}{section}

\def\bq{\begin{equation}}
\def\eq{\end{equation}}
\def\br{\begin{eqnarray}}
\def\er{\end{eqnarray}}
\def\brr{\bq\begin{array}{rlll}}
\def\err{\end{array}\eq}

\def\pmb#1{\mbox{\boldmath $#1$}}\def\text#1{\hbox{#1}}
\newtheorem{theorem}{Theorem}[section]

\newcommand{\p}{\pmb{p}}

\newcommand{\eps}{\epsilon}

\def\bfE{\mbox{\boldmath$E$}}

\def\bfJ{\mbox{\boldmath$J$}}
\def\bfu{\mbox{\boldmath$u$}}

\def\bfv{\mbox{\boldmath$v$}}
\def\bfV{\mbox{\boldmath$V$}}

\def\bfphi{\mbox{\boldmath${\bf\phi}$}}

\def\bfchi{\mbox{\boldmath${\bf\chi}$}}
\def\bfx{\mbox{\boldmath$x$}}
\def\bfp{\mbox{\boldmath$p$}}

\def\bff{\mbox{\boldmath$f$}}

\def\pa{\partial}
\def\eps{\epsilon}
\def\O{\Omega}
\def\ov{\overline}
\def\e0{\varepsilon_0}

\def\bfV{\mbox{\boldmath$V$}}

\title{ A novel finite element method for simulating surface plasmon polaritons on complex graphene sheets 
}

\author{
Jichun Li \thanks{Department of Mathematical Sciences,
 University of Nevada Las Vegas, Nevada 89154-4020, USA
({\tt jichun.li@unlv.edu})}
\and Michael Neunteufel  \thanks{The Fariborz Maseeh Department of Mathematics and Statistics,
Portland State University, Portland, OR 97021, USA ({\tt  mneunteu@pdx.edu}).}
\and
 Li Zhu \thanks{The Fariborz Maseeh Department of Mathematics and Statistics,
Portland State University, Portland, OR 97021, USA ({\tt
   lizhu@pdx.edu}).}
}

\begin{document}
\maketitle

\abstract{Surface plasmon polaritons (SPPs) are generated on the graphene surface, and
  provide a window into the nano-optical and electrodynamic response of their host material and its dielectric environment. An accurate simulation of SPPs presents several unique challenges, since SPPs often occur at complex interfaces between materials of different dielectric constants and appropriate boundary conditions at the graphene interfaces are crucial. Here we develop a simplified graphene model and propose a new finite element method accordingly. Stability for the continuous model is established, and extensive numerical results are presented to demonstrate that the new model can capture the SPPs very well for various complex graphene sheets.
}



\section{Introduction}
Graphene is a single layer of carbon atoms arranged in a hexagonal lattice pattern, often described as a "honeycomb" structure. The 2-D material graphene was first successfully isolated by Novoselov and  Geim {\it et al.}  \cite{Novoselov}.
In 2010, Geim and Novoselov were awarded Nobel Prizes in Physics for their groundbreaking 
 experiments regarding graphene.
Due to graphene's unique electrical, electromagnetic,
and optical characteristics, it has attracted widespread attention, leading to the design of
many new systems and equipment with graphene. 
For example, graphene has
played a prominent role in the design of organic light-emitting
diodes, solar cells, antennas, and invisibility cloaks (cf. \cite{gra1,gra2,Graspp}).

In the past two decades, in addition to many great achievements on the physical research on graphene and graphene-based devices, another hot topic has been the numerical simulation of graphene.
The famous Kubo formula \cite{Hanson_2008} gives the expression of graphene conductivity, which is a function
of many physical parameters such as wavelength, chemical potential, and temperature.
With the development of
computational electromagnetic in the past, such as the finite difference time-domain (FDTD) method  (e.g., \cite{Hong,Huang_ADI,Jenkinson_JCAM2018, WLi2013,FDTDbk2}) and
the finite element method (FEM)  (e.g., papers \cite{Boffi, Buffa,Carstensen,ChenDuZou, Li_JCP2014, Scheid,Huang_SISC2013,Shi_JCAM2018,YangWang}, and books \cite{Dembook2,Monk,Libook}), many robust and efficient numerical methods 
 have been proposed to simulate the electromagnetic response of graphene related devices.
The numerical approaches for modeling graphene can be classified into two big categories:
 (1) Treating graphene as a thin plate with finite thickness \cite{Rickhaus_2020}, and
converting its surface conductivity into a volumetric conductivity;
 (2) Taking graphene as a zero-thickness sheet \cite{Nayyeri}.
Due to the
easy realization of (1), many published papers and commercial software model graphene
as a thin sheet with a finite thickness \cite{Bouzianas,wang2015broadband}. However, direct discretization of graphene
with a small finite thickness results in extremely fine grids around graphene. This leads to extremely small time steps
for time-domain simulations with explicit schemes,  which
consume enormous amount of memory storage and CPU time.

The interesting physical research on graphene and graphene-based
devices has inspired  mathematicians  to conduct mathematical analysis  \cite{Bal_2023,Hong_2021,Lee_Weistein_2019} and modeling of graphene 
(e.g., \cite{Kong_SIAP2024,Maier_JCP2017,Song_CMAME2019,Santosa_SIAM2019} and references therein).
Here we are interested in simulating the surface plasmon polaritons (SPPs) generated on the graphene surface,
since SPPs  provide a window into the nano-optical, electrodynamic response of their host material and its dielectric environment \cite{Vitalone_ACS2024}. 
The plasmons occur in the highly sought after terahertz to mid-infrared regime. Note that
terahertz waves are used in a variety of applications such as nondestructive analysis, since terahertz waves can be used to analyze the internal structure of objects without damaging them. For example, THz cameras can be used to see what is inside sealed packages. Also, terahertz waves can be found in military applications, e.g.,
terahertz sensors are used in military communication, detecting biological and explosive warfare agents, and inspecting concealed weapons. However, because of practical difficulties in exciting and detecting the SPP waves in
graphene, numerical simulation of wave interactions with graphene materials plays a very important role in designing functional components with  graphene. 
Considering the disadvantage of FDTD methods in handling the complex geometry, which happens quite often in graphene devices, 
we recently  proposed and analyzed some finite element time-domain (FETD) methods for graphene simulation \cite{Yang_CMAME2020,Li_RINAM2021,Huang_CMA2022,Li_CMAWA2023}.
In  \cite{Yang_CMAME2020,Li_RINAM2021}, we treated the graphene with some thickness (though very thin); while in \cite{Li_CMAWA2023}, we successfully developed a new FETD method for  simulating surface plasmon polaritons (SPPs) on graphene sheets with zero-thickness.
We like to remark that the works of \cite{Maier_JCP2017,Song_CMAME2019} are based on frequency-domain finite element methods. 
The work of Wilson, Santosa, and Martin \cite{Santosa_SIAM2019} is for
 2-D time-domain SPP models written as an integro-differential equation in time. In \cite{Nicholls_JOSA2021}, Nicholls {\it et al.} proposed  the so-called
high–order perturbation of surfaces algorithms for simulating the plasmonic response of a perfectly flat sheet of graphene. They solved the Helmholtz equations by using the dispersive Drude model for the
surface current.

In this paper, we propose a novel FETD method by treating the graphene  sheet as zero-thickness.
The new method is based on a reformulated system of governing equations for graphene with electric and magnetic fields as unknowns. 
Compared to our previous work \cite{Li_CMAWA2023}, the new system does not contain the induced current explicitly. Hence, our new method is more efficient in memory and computational cost.

The rest of the paper is organized as follows. In Section 2, we first
present and reformulate the time-domain governing equations for modeling the surface plasmon polaritons
on graphene sheets. Then we prove the stability for the model. In Section 3, we propose a leapfrog 
time stepping finite element scheme for solving the graphene model. 
In Section 4, we present extensive numerical results to demonstrate the effectiveness of our method for simulating the propagation of surface plasmon polaritons on various complex graphene sheets.
 We conclude the paper in Section 5.

\section{The governing equations and stability analysis}

In our previous papers \cite{Yang_CMAME2020,Li_RINAM2021},
we treated the graphene as a homogenized material with an effective permittivity and a small  thickness.
By ignoring  the interband conductivity, we have the $TE_z$ model governing equations for simulating surface plasmon propagation on graphene \cite[(2.7)-(2.12)]{Yang_CMAME2020}:
\begin{eqnarray}
    && \eps_0\pa_t \bfE = \nabla \times H, \quad\mbox{in} ~~\O,  \label{m1} \\
    && \mu_0\pa_t H = -\nabla \times \bfE - K_s,\quad\mbox{in} ~~\O, \label{m2} \\
    && \tau_0 \pa_t \bfJ + \bfJ =\sigma_0 \bfE, \quad\mbox{on} ~\Gamma, \label{m4} 
\end{eqnarray}
where we denote the electric field  $\bfE = (E_x, E_y)'$, magnetic field $H=H_z$,
$K_s$ for an imposed magnetic source function, $\bfJ:=\bfJ_d$ (as denoted in \cite{Yang_CMAME2020}) 
for the induced intraband surface current in graphene,
 $\eps_0$ and $\mu_0$ for  the  vacuum permittivity and permeability, 
  the positive constant $\tau_0$ for the relaxation time,
     and the positive constant $\sigma_0$ for the graphene surface conductivity. 
Moreover, we assume that the physical domain $\O$ is a bounded Lipschitz polygonal domain in ${\cal R}^2$ with boundary $\pa\O$, and $\Gamma$ represents the graphene sheet embedded in $\O$. 
Also we define the 2-D curl operators  as
   $\nabla \times H:=(\pa_y H, -\pa_x H)'$ and $\nabla \times \bfE:=\pa_x E_y - \pa_y E_x$.

To complete the model,
 we assume that (\ref{m1})-(\ref{m4}) are subject to the simple perfectly  conducting (PEC) boundary condition:
\begin{equation}
\boldsymbol{\hat{\nu}} \times \bfE = 0,~~ \mbox{on} ~\partial\Omega, 
\label{PEC}
\end{equation}
and the initial conditions
\begin{eqnarray}
 \bfE(\bfx,0)=\bfE_0(\bfx), ~H(\bfx,0)=H_0(\bfx),
~\bfJ(\bfx,0)=\bfJ_{0}(\bfx),   
\label{IC_mod}
\end{eqnarray}
where $\hat{\boldsymbol{\nu}}$ denotes the unit outward normal vector on $\pa\Omega$,
 and $\bfE_0, H_0, \bfJ_{0}$ are some properly given functions.

Recently, we \cite{Li_CMAWA2023} successfully simulated surface plasmon propagation on graphene by treating it 
as zero-thickness sheet (i.e., appearing as a curve in our 2-D simulations, e.g., Figures \ref{bifurcate_straight_mesh} and \ref{bifurcate_curve_mesh} shown later),
with the following  boundary conditions given on the graphene interface  \cite[Fig.1]{Bludov_2013}:
  \begin{eqnarray}
&& \hat{n}_1\times {\bfE}_1 = \hat{n}_2\times {\bfE}_2, ~~ \mbox{on}~ \Gamma, \label{bc1} \\
&& H_1 - H_2 = \bfJ\times \boldsymbol{\hat{n}}, ~~ \mbox{on}~ \Gamma, \label{bc2} 
   \end{eqnarray}
where $H_1$ and $H_2$ represent the magnetic fields above and below the interface, respectively, 
$\boldsymbol{\hat{n}}:=(n_x,n_y)'$ denotes the unit normal vector of the interface pointing upward, and $\hat{n}_1$ and
   $\hat{n}_2$ are the unit outward normal vectors from the top and bottom
   of the interface. Finally,  we adopt the 2-D cross product notation
   $\bfJ\times \boldsymbol{\hat{n}}:=J_xn_y-J_yn_x$. We like to remark that  \eqref{bc1}-\eqref{bc2} imply that the tangential electric field is  continuous across the interface,
   and the jump of the tangential component of the magnetic field across the interface is equal to the surface current.

In \cite{Li_CMAWA2023}, using integration by parts and the interface conditions \eqref{bc1}-\eqref{bc2},
we derived  the following weak formulation for simulating graphene as an interface: 
 Find the solution 
 \begin{eqnarray*}
 \bfE\in L^2(0,T; H_0(curl;\O))\cap  H^1(0,T; (L^2(\O))^2), 
 H \in H^1(0,T; L^2(\O)),
 \bfJ \in   H^1(0,T; (L^2(\Gamma))^2),
 \label{ass1}
\end{eqnarray*}
 such that (cf. \cite[(2.9)-(2.11)]{Li_CMAWA2023}):
\begin{eqnarray}
&& \eps_0 (\pa_t \bfE, \bfphi) = ( H, \nabla\times \bfphi) - \langle\bfJ, \bfphi\rangle_{\Gamma}, \label{m5} \\
&& \mu_0 (\pa_t H, \psi) = -(\nabla\times \bfE, \psi) -  (K_s, \psi), \label{m6} \\
&& \langle\tau_0 \pa_t \bfJ, \bfchi\rangle_{\Gamma} +  \langle\bfJ, \bfchi\rangle_{\Gamma}= \langle \sigma_0\bfE, \bfchi\rangle_{\Gamma}, \label{m7}
   \end{eqnarray}
hold true for any test functions $\bfphi \in H_0(curl;\Omega), \psi \in L^2(\Omega)$,
and $\bfchi \in (L^2(\Gamma))^2$. To obtain (\ref{m5}), we use integration by parts over $\O$ and  the boundary condition (\ref{bc2}).
Here  we denote $(\cdot,\cdot)$ for the $L^2$ inner product over
$\O$, and $\langle\bfJ,
\bfphi\rangle_{\Gamma}:=\int_{\Gamma}\bfJ\times \boldsymbol{\hat{n}} \cdot
\bfphi\times \boldsymbol{\hat{n}}~ds$ for the inner product on $\Gamma$.
Finally, we adopt  the standard Sobolev space notation 
$$H_0(curl;\O)=\{ \bfu\in (L^2(\O))^2:~ \nabla\times \bfu \in
L^2(\O), ~\boldsymbol{\hat{\nu}} \times \bfu = 0~ \mbox{on} ~\pa\O  \}.$$

To derive the governing equations for our new numerical method, we differentiate \eqref{m5} with respect to $t$ and then multiplying the result by $\tau_0 $. This gives
\begin{eqnarray}
&& \tau_0 \eps_0 (\pa_{tt} \bfE, \bfphi) = \tau_0 ( \pa_tH, \nabla\times \bfphi) - \tau_0 \langle\pa_t\bfJ, \bfphi\rangle_{\Gamma}. \label{p1} 
   \end{eqnarray}

Adding \eqref{p1} and \eqref{m5} together, and using \eqref{m7} with $\bfchi=\bfphi$, we have
\begin{eqnarray}
&& \tau_0 \eps_0 (\pa_{tt} \bfE, \bfphi) +  \eps_0 (\pa_{t} \bfE, \bfphi) 
= \tau_0 ( \pa_tH, \nabla\times \bfphi) + (H, \nabla\times \bfphi)  - \sigma_0 \langle\bfE, \bfphi\rangle_{\Gamma}. \label{p3} 
   \end{eqnarray}

To develop a more efficient numerical method later, we replace $\pa_t H$ in \eqref{p3} by \eqref{m6} with $\psi = \nabla \times \bfphi$ and obtain the following new weak formulation:  Find the solution 
 $\bfE\in L^2(0,T; H_0(curl;\O))\cap  H^2(0,T; (L^2(\O))^2)$, 
 $H \in H^1(0,T; L^2(\O))$,
 such that 
\begin{eqnarray}
&&  \tau_0 \eps_0 (\pa_{tt} \bfE, \bfphi) +  \eps_0 (\pa_{t} \bfE, \bfphi)
+  \frac{\tau_0}{\mu_0} (\nabla\times \bfE, \nabla\times \bfphi)  \nonumber \\
&&\quad\quad =  (H, \nabla\times \bfphi) -  \frac{\tau_0}{\mu_0} (K_s, \nabla\times \bfphi)  - \sigma_0 \langle\bfE, \bfphi\rangle_{\Gamma}, \label{p5} \\
&& \mu_0 (\pa_t H, \psi) = -(\nabla\times \bfE, \psi) -  (K_s, \psi), \label{p6} 
   \end{eqnarray}
hold true for any test functions $\bfphi \in H_0(curl;\Omega)$ and $\psi \in L^2(\Omega)$.
The initial conditions for the problem \eqref{p5}-\eqref{p6} are as follows:
\begin{eqnarray}
 \bfE(\bfx,0)=\bfE_0(\bfx), ~H(\bfx,0)=H_0(\bfx),
~\pa_t\bfE(\bfx,0)=\eps_0^{-1}  \nabla \times H_0(\bfx).   
\label{IC2}
\end{eqnarray}

To simplify the notation,  we denote the $L^2$ norm of $u$ in $\O$
as $||u||:=||u||_{L^2(\O)}$, and the $L^2$ norm of $\bfu$ on $\Gamma$
as $||\bfu||_{\Gamma}:=(\int_{\Gamma}|\bfu\times \boldsymbol{\hat{n}}|^2~ds)^{1/2}$.

\begin{theorem} 
Denote the energy
\begin{eqnarray}
 ENG(t)  &=&  \tau_0\eps_0\|\pa_t\bfE\|^2  +\sigma_0\|\bfE\|^2_{\Gamma}
+ \mu_0\|\tau_0^{1/2}\pa_tH + \tau_0^{-1/2}H\|^2 + \frac{\mu_0}{\tau_0}\|H\|^2   \nonumber \\
&& \quad  + \frac{\eps_0}{\tau_0}\|\bfE\|^2  
+\frac{1}{\sigma_0}\|\bfJ\|^2_{\Gamma} + \tau_0\mu_0\|\pa_tH\|^2 
+\frac{\tau_0^2}{\sigma_0}\|\pa_t\bfJ\|^2_{\Gamma}.
\label{energy_cont}
\end{eqnarray} 
Then we have the following continuous stability:
\begin{eqnarray}
ENG(t) \leq \left[ ENG(0) + \int_0^t (\frac{\tau_0}{\mu_0}\|\pa_{t}K_s\|^2 +  \frac{1}{\tau_0\mu_0}\|K_s\|^2)dt \right]\cdot \exp(C_*t), \quad \forall~t\in [0,T],
\label{stab}
\end{eqnarray}
where the constant $C_*>0$ depends only on the parameter $\tau_0$.
\label{thm2}
\end{theorem}

\begin{proof} To make our proof easy to follow, we split it into several major parts.

(I)~ Choosing $\bfphi=\pa_t\bfE$ in  
\eqref{p5},  we have
\begin{eqnarray}
&& \frac{1}{2} \frac{d}{dt} \left( \tau_0\eps_0\|\pa_t\bfE\|^2 + \frac{\tau_0}{\mu_0}\|\nabla\times \bfE\|^2 
+\sigma_0\|\bfE\|^2_{\Gamma} \right)
+ \eps_0\|\pa_t\bfE\|^2 \nonumber \\
&&\quad\quad = (H,\nabla\times \pa_t\bfE) - \frac{\tau_0}{\mu_0}(K_s,\nabla\times \pa_t\bfE).
\label{p8}
\end{eqnarray}

Taking the time derivative of \eqref{m2}, we obtain
\begin{eqnarray}
\nabla \times \pa_t\bfE =  -  \mu_0\pa_{tt} H  - \pa_t K_s. \label{p9} 
\end{eqnarray}

Replacing $\nabla \times \pa_t\bfE$ in \eqref{p8} by \eqref{p9}, we have
\begin{eqnarray}
&& \frac{1}{2} \frac{d}{dt} \left( \tau_0\eps_0\|\pa_t\bfE\|^2 + \frac{\tau_0}{\mu_0}\|\nabla\times \bfE\|^2 
+\sigma_0\|\bfE\|^2_{\Gamma} \right)
+ \eps_0\|\pa_t\bfE\|^2 \nonumber \\
&&\quad\quad = -\mu_0(H,\pa_{tt}H) - (H,\pa_{t}K_s)
+ \tau_0(K_s,\pa_{tt}H) + \frac{\tau_0}{\mu_0}(K_s, \pa_tK_s).
\label{p10}
\end{eqnarray}

(II)~ By choosing $\psi=\tau_0\pa_{tt}H$ in   \eqref{p6},  then 
replacing $\pa_{tt}H$ by \eqref{p9}, we have
\begin{eqnarray}
 \frac{1}{2}\frac{d}{dt} \left( \tau_0 \mu_0\|\pa_tH\|^2 \right)
&=& -\tau_0(\nabla\times \bfE, \pa_{tt}H) - \tau_0 (K_s,\pa_{tt}H) \nonumber \\
&=& \frac{\tau_0}{\mu_0}(\nabla\times \bfE, \nabla\times \pa_t\bfE + \pa_{t}K_s) - \tau_0 (K_s,\pa_{tt}H) \nonumber \\
&=& \frac{\tau_0}{2\mu_0}\frac{d}{dt} (||\nabla\times \bfE||^2) 
+  \frac{\tau_0}{\mu_0}( \nabla\times \bfE, \pa_{t}K_s) - \tau_0 (K_s,\pa_{tt}H).
\label{p11}
\end{eqnarray}

Adding \eqref{p10} and \eqref{p11} together, and using \eqref{m2}, we obtain
\begin{eqnarray}
&& \frac{1}{2} \frac{d}{dt} \left( \tau_0\eps_0\|\pa_t\bfE\|^2 + \tau_0\mu_0\|\pa_tH\|^2 
+\sigma_0\|\bfE\|^2_{\Gamma} \right)
+ \eps_0\|\pa_t\bfE\|^2 \nonumber \\
&=& -\mu_0(H,\pa_{tt}H) - (H,\pa_{t}K_s) - \tau_0(\pa_{t}H, \pa_tK_s)   \nonumber \\
&=& -\mu_0\left[ \frac{d}{dt}(H,\pa_{t}H) - (\pa_{t}H,\pa_{t}H)\right]
- (H,\pa_{t}K_s) - \tau_0(\pa_{t}H, \pa_tK_s),
\label{p12}
\end{eqnarray}
which can be rewritten as
\begin{eqnarray}
&& \frac{1}{2} \frac{d}{dt} \left( \tau_0\eps_0\|\pa_t\bfE\|^2 +\sigma_0\|\bfE\|^2_{\Gamma}
+ \mu_0\|\tau_0^{1/2}\pa_tH + \tau_0^{-1/2}H\|^2 - \frac{\mu_0}{\tau_0}\|H\|^2 
\right)
+ \eps_0\|\pa_t\bfE\|^2 \nonumber \\
&=& \mu_0||\pa_{t}H||^2 - (H,\pa_{t}K_s) - \tau_0(\pa_{t}H, \pa_tK_s).
\label{p13}
\end{eqnarray}

(III)~To bound the terms $H$ and $\pa_{t}H$ on the right-hand side (RHS) of \eqref{p13},
we first choose $\bfphi=\bfE, \psi=H, \bfchi=\frac{1}{\sigma_0}\bfJ$ in  
(\ref{m5})-(\ref{m7}), respectively. Adding the results together, we obtain
\begin{eqnarray}
\frac{1}{2} \frac{d}{dt} \left( \eps_0\|\bfE\|^2 + \mu_0\|H\|^2 
+\frac{\tau_0}{\sigma_0}\|\bfJ\|^2_{\Gamma} \right)
+ \frac{1}{\sigma_0}\|\bfJ\|^2_{\Gamma} = -(K_s,H).
\label{p14}
\end{eqnarray} 

Multiplying \eqref{p14} by $\frac{2}{\tau_0}$, we have
\begin{eqnarray}
\frac{1}{2} \frac{d}{dt} \left( \frac{2\eps_0}{\tau_0}\|\bfE\|^2 + \frac{2\mu_0}{\tau_0}\|H\|^2 
+\frac{2}{\sigma_0}\|\bfJ\|^2_{\Gamma} \right)
+ \frac{2}{\tau_0\sigma_0}\|\bfJ\|^2_{\Gamma} = -\frac{2}{\tau_0}(K_s,H).
\label{p15}
\end{eqnarray} 

Similarly, taking the time derivative of (\ref{m5})-(\ref{m7}), then choosing $\bfphi=\pa_t\bfE, \psi=\pa_tH, \bfchi=\frac{1}{\sigma_0}\pa_t\bfJ$ respectively, and
 adding the results together, we obtain
\begin{eqnarray}
\frac{1}{2} \frac{d}{dt} \left( \eps_0\|\pa_t\bfE\|^2 + \mu_0\|\pa_tH\|^2 
+\frac{\tau_0}{\sigma_0}\|\pa_t\bfJ\|^2_{\Gamma} \right)
+ \frac{1}{\sigma_0}\|\pa_t\bfJ\|^2_{\Gamma} = -(\pa_tK_s,\pa_tH).
\label{p16}
\end{eqnarray} 

Multiplying \eqref{p16} by $\tau_0$, we have
\begin{eqnarray}
\frac{1}{2} \frac{d}{dt} \left( \tau_0\eps_0\|\pa_t\bfE\|^2 + \tau_0\mu_0\|\pa_tH\|^2 
+\frac{\tau_0^2}{\sigma_0}\|\pa_t\bfJ\|^2_{\Gamma} \right)
+ \frac{\tau_0}{\sigma_0}\|\pa_t\bfJ\|^2_{\Gamma} = -\tau_0(\pa_tK_s,\pa_tH).
\label{p17}
\end{eqnarray} 

Now adding \eqref{p13}, \eqref{p15}, and \eqref{p17} together, we have
\begin{eqnarray}
&& \frac{1}{2} \frac{d}{dt} \left(2 \tau_0\eps_0\|\pa_t\bfE\|^2  +\sigma_0\|\bfE\|^2_{\Gamma}
+ \mu_0\|\tau_0^{1/2}\pa_tH + \tau_0^{-1/2}H\|^2 + \frac{\mu_0}{\tau_0}\|H\|^2  \right. \nonumber \\
&& \quad \left. + \frac{2\eps_0}{\tau_0}\|\bfE\|^2  
+\frac{2}{\sigma_0}\|\bfJ\|^2_{\Gamma} + \tau_0\mu_0\|\pa_tH\|^2 
+\frac{\tau_0^2}{\sigma_0}\|\pa_t\bfJ\|^2_{\Gamma} \right)  \nonumber \\
&& \quad + \eps_0\|\pa_t\bfE\|^2 + \frac{2}{\tau_0\sigma_0}\|\bfJ\|^2_{\Gamma}
+ \frac{\tau_0}{\sigma_0}\|\pa_t\bfJ\|^2_{\Gamma} \nonumber \\
&=&  \mu_0||\pa_{t}H||^2 - (H,\pa_{t}K_s) - 2\tau_0(\pa_{t}H, \pa_tK_s) 
-\frac{2}{\tau_0}(K_s,H).
\label{p18}
\end{eqnarray} 

(IV)~ Using the Young's inequality, we can bound the last three RHS terms of \eqref{p18} as follows:
\begin{eqnarray}
&&  - (H,\pa_{t}K_s) \leq   \frac{\mu_0}{2\tau_0}\|H\|^2 +  \frac{\tau_0}{2\mu_0}\|\pa_{t}K_s\|^2,  \label{p19}\\
&& - 2\tau_0(\pa_{t}H, \pa_tK_s) \leq 2\tau_0\mu_0\|\pa_tH\|^2  +  \frac{\tau_0}{2\mu_0}\|\pa_{t}K_s\|^2,  \label{p20}\\
&& -\frac{2}{\tau_0}(K_s,H) \leq  \frac{\mu_0}{2\tau_0}\|H\|^2 +  \frac{1}{2\tau_0\mu_0}\|K_s\|^2.
\label{p21}
\end{eqnarray} 

Substituting the above estimates \eqref{p19}-\eqref{p21} into \eqref{p18}, and dropping the last three non-negative terms on the left hand side of \eqref{p18}, we obtain
\begin{eqnarray}
&& \frac{1}{2} \frac{d}{dt} \left(2 \tau_0\eps_0\|\pa_t\bfE\|^2  +\sigma_0\|\bfE\|^2_{\Gamma}
+ \mu_0\|\tau_0^{1/2}\pa_tH + \tau_0^{-1/2}H\|^2 + \frac{\mu_0}{\tau_0}\|H\|^2  \right. \nonumber \\
&& \quad \left. + \frac{2\eps_0}{\tau_0}\|\bfE\|^2  
+\frac{2}{\sigma_0}\|\bfJ\|^2_{\Gamma} + \tau_0\mu_0\|\pa_tH\|^2 
+\frac{\tau_0^2}{\sigma_0}\|\pa_t\bfJ\|^2_{\Gamma} \right)  \nonumber \\
&\leq&  \tau_0\mu_0(2+ \tau_0^{-1})||\pa_{t}H||^2 + \frac{\mu_0}{\tau_0}\|H\|^2 
+  \frac{\tau_0}{\mu_0}\|\pa_{t}K_s\|^2 +  \frac{1}{2\tau_0\mu_0}\|K_s\|^2.
\label{p23}
\end{eqnarray} 

The proof is complete by applying the Gronwall inequality to \eqref{p23}.
\end{proof}

 \section{The leapfrog finite element scheme and its analysis}

   To design a finite element method, we partition the physical domain 
$\Omega$ with $\Gamma$ as an internal boundary 
by a  shape regular  triangular mesh $\mathcal {T}_h$ with maximum
mesh size $h$. Without loss of generality, we 
consider the following 
Raviart-Thomas-N\'{e}d\'{e}lec (RTN)'s mixed spaces $U_h$ and
$\bfV_h$ on triangular elements \cite{Libook,Monk}: For any $r\geq 1$,
\begin{align*}
         & U_h = \{u_h \in L^2(\O):~ u_h|_K \in p_{r-1}, \forall K \in T_h\},\\
         & \bfV_h = \{\bfv_h \in H(curl;\O):~ \bfv_h|_K \in (p_{r-1})^2 \oplus S_r, \forall K \in T_h\}, ~~ S_r= \{\bfp \in \Tilde{p}_r^2, \bfx \cdot \bfp=0 \},
 \end{align*}
where $p_r$ denotes  the space of polynomials
of degree less than or equal to $r$, and $\Tilde{p}_r^2$ represents  the space of homogeneous polynomials of degree $r$.

 To handle the PEC boundary condition (\ref{PEC}), we introduce the subspace
$$\bfV_h^0=\{\bfv_h \in \bfV_h:~~ \boldsymbol{\hat{\nu}}\times \bfv_h
=0 \quad\mbox{on}~\partial\Omega \}.$$

To construct the fully discrete finite element scheme, we partition the time interval $[0,T]$ uniformly by points $t_i=i\tau, i=0,...,N_t$, where $\tau=\frac{T}{N_t}$ denotes the time step size.

 Now we can construct the following leapfrog time stepping scheme:
 Given proper initial approximations of $\bfE_h^{-1}, \bfE_h^0 \in \bfV_h$, $H_h^{\frac{1}{2}}\in U_h$, for any $n \ge 0$, find  $\bfE_{h}^{n+1} \in \bfV_h$, $H_h^{n+\frac{1}{2}}\in U_h$ such that
 \begin{eqnarray}
 && \eps_0 (\delta_{\tau}^2\bfE_h^{n}, \bfphi_h) 
 + \frac{\eps_0}{\tau_0}(\delta_{2\tau}\bfE_h^{n}, \bfphi_h) +\frac{1}{\mu_0} (\nabla \times \bfE^n_h,\nabla \times \bfphi_h) \nonumber \\
&&  \hskip 1in = \frac{1}{\tau_0}(\ov{H}^{n}_{h}, \nabla\times \bfphi_h) - \frac{\sigma_0}{\tau_0} \langle\bfE_h^n, \bfphi_h\rangle_{\Gamma} - \frac{1}{\mu_0}(K_s^n, \nabla \times \bfphi_h), \label{sc1} \\
 && \mu_0 (\delta_{\tau}H^{n}_{h}, \psi_h) =   - (\nabla \times \bfE_{h}^{n}, \psi_h) -(K_{s}^{n}, \psi_h), \label{sc2} 
\end{eqnarray}
hold true  for any test functions $\bfphi_h \in \bfV_h^0, \psi_h \in U_h$.
Here we adopt the following central difference operators and averaging operator  in time: For any time sequence function $u^n$,
$$ 
\delta_{\tau}u^{n}=\frac{u^{n+\frac{1}{2}}-u^{n-\frac{1}{2}}}{\tau},
~ \delta_{2\tau}u^{n}=\frac{u^{n+1}-u^{n-1}}{2\tau},
~ \delta_{\tau}^2u^{n}=\frac{u^{n+1}-2u^{n}+u^{n-1}}{\tau^2},
~ 
 \ov{u}^{n}=\frac{u^{n+\frac{1}{2}}+u^{n-\frac{1}{2}}}{2}.
$$

Corresponding to  the finite element spaces $\bfV_h$ and $U_h$,
 we denote 
 $\Pi_c$ and $\Pi_2$ for the standard N\'{e}d\'{e}lec interpolation
 in space $\bfV_h$ and the standard $L^2$ projection onto
 space $U_h$, respectively. 
 Furthermore,  the following interpolation and projection errors hold true (cf. \cite{Libook,Monk}):
 \begin{eqnarray}
&& ||\bfu-\Pi_c\bfu||_{H(curl;\O)}\leq ch^r||\bfu||_{H^r(curl;\O)},
~~\forall~\bfu\in H^r(curl;\O),\quad r\geq1,\label{Int}  \\
&& ||u-\Pi_2u||_{L^2(\O)}\leq ch^r||u||_{H^r(\O)},
~~\forall~u\in H^r(\O),\quad r\geq0,\label{Pro}
\end{eqnarray}
where $||u||_{H^r(\O)}$ denotes the norm for the Sobolev space $H^r(\O)$, and 
$||\bfu||_{H^r(curl;\O)}:=(||\bfu||^2_{(H^r(\O))^2}+||\nabla\times \bfu||^2_{H^r(\O)})^{1/2}$ is the norm for the Sobolev space 
$$H^r(curl;\O)=\{\bfu\in (H^r(\O))^2: ~\nabla\times \bfu \in H^r(\O)\}.$$

 The initial conditions  (\ref{IC2}) are discretized as follows:
\begin{align}
  &\bfE_h^{0}=\Pi_c\bfE_0(\bfx),  \label{lf_in0}  \\
  &H_h^{-\frac{1}{2}}=\Pi_2\left[H(\cdot,0) -\frac{\tau}{2}\pa_tH(\cdot,0)\right]
= \Pi_2\left[ H_0(\bfx)+\frac{\tau}{2\mu_0}(\nabla\times \bfE_0(\bfx)+K_s(\bfx,0)) \right],    \label{lf_in1}  \\
&\frac{\bfE_h^{1} - \bfE_h^{-1}}{2\tau}=\Pi_c\left(\eps_0^{-1}\nabla \times H_0(\bfx)\right),
\label{lf_in2}
\end{align}
where we used the Taylor expansion and the governing equation (\ref{m2}).

The implementation of the scheme \eqref{sc1}-\eqref{sc2} is quite simple. At each time step, we first solve \eqref{sc2} for $H_{h}^{n+\frac{1}{2}}$; then solve \eqref{sc1} for $\bfE_{h}^{n+1}$.
Of course, at the first time step  when $n=0$, we need to use the initial conditions \eqref{lf_in0}-\eqref{lf_in2}.
  It can be seen that the coefficient matrix at each time step is symmetric and positive definite. Hence, the solvability of our scheme \eqref{sc1}-\eqref{sc2} is guaranteed. 

\subsection{The stability analysis}
To prove the discrete stability for our scheme \eqref{sc1}-\eqref{sc2},
we denote $C_v=\frac{1}{\sqrt{\epsilon_0\mu_0}}$ for the wave propagation speed,
and recall the standard inverse estimate: 
\begin{align}
 ||\nabla\times \bfphi_h|| \leq C_{\mathrm{in}}h^{-1}||\bfphi_h||, \quad
\forall~\bfphi_h \in \bfV_h,  
\label{inv}
\end{align}
and the trace inequality:
\begin{align}
 ||\bfphi_h||_{\Gamma} \leq C_{\mathrm{tr}}h^{-1/2}||\bfphi_h||, 
 \quad  \forall~\bfphi_h \in \bfV_h,  
\label{trace}
\end{align}
where the positive constants $C_{\mathrm{in}}$ and $C_{\mathrm{tr}}$ are independent of the mesh size $h$.

\begin{theorem}
\label{thm2}
Denote the discrete energy  for the scheme \eqref{sc1}-\eqref{sc2}:
\begin{eqnarray}
ENG_{m} &:=& \eps_0||\delta_{\tau}\bfE_h^{m+\frac{1}{2}}||^2 
+ \frac{1}{2\mu_0}(||\nabla \times \bfE^{m+1}_h||^2 + ||\nabla \times \bfE^{m}_h||^2) 
+ \mu_0||H_h^{m+\frac{1}{2}}||^2    \nonumber \\
&&   + \frac{\sigma_0}{2\tau_0}(||\bfE_{h}^{m+1}||^2_{\Gamma}+||\bfE_{h}^{m}||^2_{\Gamma}) 
+ \frac{\tau}{4\mu_0\tau_0}(||\nabla \times \bfE^{m+1}_h||^2 + ||\nabla \times \bfE^{m}_h||^2). 
\label{ENG_lf}
\end{eqnarray}
Then under the time step constraint:
\begin{eqnarray}
\tau \leq \min\left(1,   \frac{h}{2C_{\mathrm{in}}C_v},  \frac{h\sqrt{\eps_0\tau_0}}{2C_{\mathrm{tr}}\sqrt{\sigma_0}},   \frac{h\sqrt{\tau_0}}{\sqrt{2}C_{\mathrm{in}}C_v},   \frac{h\tau_0}{C_{\mathrm{in}}C_v}\right),
\label{CFL}
\end{eqnarray}
we have the following stability: For any $m \in [1, N_t-1]$,
\begin{eqnarray}
ENG_{m} \leq C_{*}\left[ENG_{0} + ||K_s^{0}||^2+||K_s^{1}||^2+\tau\sum_{n=1}^{m}
(||K_s^{n}||^2 +  ||\delta_{2\tau}K_s^{n}||^2)\right],
 \label{stab_lf}
\end{eqnarray}
where the positive constant $C_{*}$ is independent of $\tau$ and $h$.
\end{theorem}

\begin{proof} To make our proof easy to follow, we divite it into several major steps.

(I)~ Choosing $\bfphi_h=\tau \delta_{2\tau}\bfE_h^{n}$ in \eqref{sc1}, and using the following identities
\begin{eqnarray}
&& \tau ( \delta_{\tau}^2u^{n}, \delta_{2\tau}u^{n}) = \frac{1}{2}(|| \delta_{\tau}u^{n+ \frac{1}{2}}||^2
- || \delta_{\tau}u^{n- \frac{1}{2}}||^2), \label{re1}\\
&& \tau ( u^{n}, \delta_{2\tau}u^{n}) = \frac{1}{4}(|| u^{n+1}||^2
- ||u^{n-1}||^2) - \frac{\tau^2}{4}(|| \delta_{\tau}u^{n+ \frac{1}{2}}||^2
- || \delta_{\tau}u^{n-\frac{1}{2}}||^2), \label{re2}
\end{eqnarray}
with $u^n=\bfE_h^n$, we have
\begin{eqnarray}
 && \frac{\eps_0}{2} (|| \delta_{\tau}\bfE_h^{n+ \frac{1}{2}}||^2
- || \delta_{\tau}\bfE_h^{n- \frac{1}{2}}||^2)
+  \frac{\tau\eps_0}{\tau_0}||\delta_{2\tau}\bfE_h^{n}||^2 \nonumber \\
&&  +\frac{1}{4\mu_0} \left[ (||\nabla \times \bfE^{n+1}_h||^2 - ||\nabla \times \bfE^{n-1}_h||^2)
- \tau^2(||\nabla \times \delta_{\tau}\bfE_h^{n+ \frac{1}{2}}||^2 - ||\nabla \times \delta_{\tau}\bfE_h^{n-\frac{1}{2}}||^2)\right] \nonumber \\
&&  +\frac{\sigma_0}{4\tau_0} \left[ (||\bfE^{n+1}_h||^2_{\Gamma}  - ||\bfE^{n-1}_h||^2_{\Gamma} )
- \tau^2(||\delta_{\tau}\bfE_h^{n+ \frac{1}{2}}||^2_{\Gamma}  - ||\delta_{\tau}\bfE_h^{n-\frac{1}{2}}||^2_{\Gamma} )\right] \nonumber \\
&&  \hskip 1in = \frac{\tau}{\tau_0}(\ov{H}^{n}_{h}, \nabla\times \delta_{2\tau}\bfE_h^{n}) 
-  \frac{\tau}{\mu_0}(K_s^n, \nabla \times \delta_{2\tau}\bfE_h^{n}). \label{re3} 
\end{eqnarray}

Taking $\psi_h= \frac{\tau^2}{2\mu_0\tau_0}\nabla \times \delta_{2\tau}\bfE_h^{n}$
in \eqref{sc2}, and using \eqref{re2} with $u^n=\bfE_h^n$, we obtain
\begin{eqnarray}
 && \frac{\tau}{2\tau_0} (H^{n+ \frac{1}{2}}_{h} - H^{n- \frac{1}{2}}_{h}, \nabla \times \delta_{2\tau}\bfE_h^{n}) = -\frac{\tau^2}{2\mu_0\tau_0} (K_s^{n}, \nabla \times \delta_{2\tau}\bfE_h^{n})   \label{re4}   \\
&&\quad 
- \frac{\tau}{8\mu_0\tau_0} \left[ (||\nabla \times \bfE^{n+1}_h||^2 - ||\nabla \times \bfE^{n-1}_h||^2)
- \tau^2(||\nabla \times \delta_{\tau}\bfE_h^{n+ \frac{1}{2}}||^2 - ||\nabla \times \delta_{\tau}\bfE_h^{n-\frac{1}{2}}||^2)\right].   \nonumber
\end{eqnarray}

Taking $\psi_h= \tau\ov{H}_h^{n}$
in \eqref{sc2},  we obtain
\begin{eqnarray}
 \frac{\mu_0}{2} (||H^{n+ \frac{1}{2}}_{h}||^2 - ||H^{n- \frac{1}{2}}_{h}||^2)
 =    - \tau  (\nabla \times \bfE^{n}_h,\ov{H}_h^{n})  - \tau (K_s^{n}, \ov{H}_h^{n}).
   \label{re6}  
\end{eqnarray}

Adding \eqref{re3}, \eqref{re4}, and \eqref{re6} together, then summing up the result from $n=1$ to
any $m\leq N_t-1$, we have
\begin{eqnarray}
 && \frac{\eps_0}{2} (|| \delta_{\tau}\bfE_h^{m+ \frac{1}{2}}||^2
- || \delta_{\tau}\bfE_h^{\frac{1}{2}}||^2)
+  \frac{\tau\eps_0}{\tau_0}\sum_{n=1}^{m}||\delta_{2\tau}\bfE_h^{n}||^2 \nonumber \\
&&  +\frac{1}{4\mu_0} \left[ (||\nabla \times \bfE^{m+1}_h||^2 + ||\nabla \times \bfE^{m}_h||^2 
- ||\nabla \times \bfE^{1}_h||^2- ||\nabla \times \bfE^{0}_h||^2)  \right.   \nonumber \\
&& \hskip 1in \left.
- \tau^2(||\nabla \times \delta_{\tau}\bfE_h^{m+ \frac{1}{2}}||^2 - ||\nabla \times \delta_{\tau}\bfE_h^{\frac{1}{2}}||^2)\right] \nonumber \\
&&  +\frac{\sigma_0}{4\tau_0} \left[ (||\bfE^{m+1}_h||^2_{\Gamma} + ||\bfE^{m}_h||^2_{\Gamma} 
- ||\bfE^{1}_h||^2_{\Gamma}- ||\bfE^{0}_h||^2_{\Gamma} )
- \tau^2(||\delta_{\tau}\bfE_h^{m+ \frac{1}{2}}||^2_{\Gamma}  - ||\delta_{\tau}\bfE_h^{\frac{1}{2}}||^2_{\Gamma} )\right] \nonumber \\
&& +  \frac{\mu_0}{2} (||H^{m+ \frac{1}{2}}_{h}||^2 - ||H^{\frac{1}{2}}_{h}||^2)   \nonumber \\
&& + \frac{\tau}{8\mu_0\tau_0} \left[ (||\nabla \times \bfE^{m+1}_h||^2 +||\nabla \times \bfE^{m}_h||^2  - ||\nabla \times \bfE^{1}_h||^2 - ||\nabla \times \bfE^{0}_h||^2)   \right.   \nonumber \\
&& \hskip 1in \left.
- \tau^2(||\nabla \times \delta_{\tau}\bfE_h^{m+ \frac{1}{2}}||^2 - ||\nabla \times \delta_{\tau}\bfE_h^{\frac{1}{2}}||^2)\right] \nonumber \\
&&  \quad = \frac{\tau}{\tau_0}\sum_{n=1}^{m}(H^{n-\frac{1}{2}}_{h}, \nabla\times \delta_{2\tau}\bfE_h^{n}) 
-  \frac{\tau}{\mu_0}\sum_{n=1}^{m}(K_s^n, \nabla \times \delta_{2\tau}\bfE_h^{n})
-   \frac{\tau^2}{2\mu_0\tau_0}\sum_{n=1}^{m}(K_s^n, \nabla \times \delta_{2\tau}\bfE_h^{n})  \nonumber\\
&& \quad\quad    - \tau\sum_{n=1}^{m}  (\nabla \times \bfE^{n}_h,\ov{H}_h^{n}) 
 - \tau \sum_{n=1}^{m}(K_s^{n}, \ov{H}_h^{n}). \label{re8}
\end{eqnarray}

(II)~ By the definition of the discrete energy $ENG_m$ and dropping the non-negative term 
$ \frac{\tau\eps_0}{\tau_0}\sum_{n=1}^{m}||\delta_{2\tau}\bfE_h^{n}||^2 $ on the left hand side of \eqref{re8}, we can rewrite \eqref{re8}  as follows
 \begin{eqnarray}
 && ENG_m \leq ENG_0 + \frac{\tau^2}{2\mu_0}(||\nabla \times \delta_{\tau}\bfE_h^{m+ \frac{1}{2}}||^2 - ||\nabla \times \delta_{\tau}\bfE_h^{\frac{1}{2}}||^2) \nonumber \\
&&\quad +  \frac{\sigma_0\tau^2}{2\tau_0}(||\delta_{\tau}\bfE_h^{m+ \frac{1}{2}}||^2_{\Gamma} 
- ||\delta_{\tau}\bfE_h^{\frac{1}{2}}||^2_{\Gamma})
+  \frac{\tau^3}{4\mu_0\tau_0}(||\nabla \times \delta_{\tau}\bfE_h^{m+ \frac{1}{2}}||^2 - ||\nabla \times \delta_{\tau}\bfE_h^{\frac{1}{2}}||^2) \nonumber \\
 &&  \quad + \frac{\tau}{\tau_0}\sum_{n=1}^{m}(H^{n-\frac{1}{2}}_{h}, \nabla\times \delta_{2\tau}\bfE_h^{n}) 
-  \frac{\tau}{\mu_0}\sum_{n=1}^{m}(K_s^n, \nabla \times \delta_{2\tau}\bfE_h^{n})
-   \frac{\tau^2}{2\mu_0\tau_0}\sum_{n=1}^{m}(K_s^n, \nabla \times \delta_{2\tau}\bfE_h^{n})  \nonumber\\
&& \quad\quad    - \tau\sum_{n=1}^{m}  (\nabla \times \bfE^{n}_h,\ov{H}_h^{n}) 
 - \tau \sum_{n=1}^{m}(K_s^{n}, \ov{H}_h^{n}) =: ENG_0 + \sum_{k=1}^{8} Sta_k. \label{re10}
\end{eqnarray}

Now we just need to estimate each $Sta_k$ in \eqref{re10}. Using the inverse estimate \eqref{inv}, it is easy to see that
\begin{eqnarray}
Sta_1 \leq \frac{\tau^2}{2\mu_0}||\nabla \times \delta_{\tau}\bfE_h^{m+ \frac{1}{2}}||^2 
\leq \frac{C_{\mathrm{in}}^2}{2\mu_0}\cdot \tau^2h^{-2}||\delta_{\tau}\bfE_h^{m+ \frac{1}{2}}||^2,
\label{re12}
\end{eqnarray}
and
\begin{eqnarray}
Sta_3 \leq \frac{\tau}{4\mu_0\tau_0}\cdot \tau^2||\nabla \times \delta_{\tau}\bfE_h^{m+ \frac{1}{2}}||^2 
\leq \frac{\tau C_{\mathrm{in}}^2}{4\mu_0\tau_0}\cdot \tau^2h^{-2}||\delta_{\tau}\bfE_h^{m+ \frac{1}{2}}||^2.
\label{re14}
\end{eqnarray}

Using the trace inequality \eqref{inv}, we have
\begin{eqnarray}
Sta_2 \leq 
\frac{\sigma_0\tau^2}{2\tau_0}||\delta_{\tau}\bfE_h^{m+ \frac{1}{2}}||^2_{\Gamma} 
\leq \frac{\sigma_0C_{\mathrm{tr}}^2}{2\tau_0}\cdot \tau^2h^{-1}||\delta_{\tau}\bfE_h^{m+ \frac{1}{2}}||^2. 
\label{re16}
\end{eqnarray}

Similarly by the inverse estimate \eqref{inv} and the Cauchy-Schwarz inequality, we have
 \begin{eqnarray}
Sta_6 &=& -\frac{\tau^2}{2\mu_0\tau_0}\sum_{n=1}^{m}(K_s^n, \nabla \times \delta_{2\tau}\bfE_h^{n}) 
\leq \frac{\tau}{2\mu_0\tau_0}\sum_{n=1}^{m}||K_s^n||\cdot C_{\mathrm{in}}\tau h^{-1}||\frac{1}{2}(\delta_{\tau}\bfE_h^{n+ \frac{1}{2}} + \delta_{\tau}\bfE_h^{n- \frac{1}{2}})|| \nonumber \\
&\leq& \tau\cdot \frac{C_{\mathrm{in}}\tau h^{-1}}{4\mu_0\tau_0}\sum_{n=1}^{m}\left[||K_s^n||^2
+ \frac{1}{2}(||\delta_{\tau}\bfE_h^{n+ \frac{1}{2}}||^2 + ||\delta_{\tau}\bfE_h^{n- \frac{1}{2}}||^2)\right].
\label{re18}
\end{eqnarray}

It is also easy to obtain the following estimates:
 \begin{eqnarray}
Sta_7 &=&  - \tau\sum_{n=1}^{m}  (\frac{1}{\sqrt{\mu_0}}\nabla \times \bfE^{n}_h,\sqrt{\mu_0}\ov{H}_h^{n}) \nonumber \\
&\leq&  \frac{\tau}{2} \sum_{n=1}^{m}  \left[\frac{1}{\mu_0}||\nabla \times \bfE^{n}_h||^2 
+ \frac{\mu_0}{2}(||H_h^{n+ \frac{1}{2}}||^2 + ||H_h^{n- \frac{1}{2}}||^2)\right],
\label{re20}
\end{eqnarray}
and 
 \begin{eqnarray}
Sta_8 \leq \frac{\tau}{2} \sum_{n=1}^{m}\left[\frac{1}{\mu_0}||K_s^{n}||^2 + \frac{\mu_0}{2}(||H_h^{n+ \frac{1}{2}}||^2 + ||H_h^{n- \frac{1}{2}}||^2)\right].
\label{re22}
\end{eqnarray}

The estimates of $Sta_4$ and $Sta_5$ need a special treatment, since $\nabla\times \delta_{2\tau}\bfE_h^{n}$ can not be controlled by $ENG_m$. In the next major part, we will carry out the estimates of  $Sta_4$ and $Sta_5$.

(III)~Note that
\begin{eqnarray}
&& \sum_{n=1}^{m}(H^{n-\frac{1}{2}}_{h}, \nabla\times \delta_{2\tau}\bfE_h^{n})
= \frac{1}{2\tau} \sum_{n=1}^{m}(H^{n-\frac{1}{2}}_{h}, \nabla\times \bfE_h^{n+1})
- \frac{1}{2\tau} \sum_{n=1}^{m}(H^{n-\frac{1}{2}}_{h}, \nabla\times \bfE_h^{n-1}) \nonumber \\
&=& \frac{1}{2\tau} \sum_{n=3}^{m+2}(H^{n-\frac{5}{2}}_{h}, \nabla\times \bfE_h^{n-1})
- \frac{1}{2\tau} \sum_{n=1}^{m}(H^{n-\frac{1}{2}}_{h}, \nabla\times \bfE_h^{n-1})
\nonumber \\
&=& -\frac{1}{2}\sum_{n=3}^{m}(\delta_{\tau}H^{n-2}_{h}+\delta_{\tau}H^{n-1}_{h}, \nabla\times \bfE_h^{n-1})
+ \frac{1}{2\tau} (H^{m-\frac{1}{2}}_{h}, \nabla\times \bfE_h^{m+1})  \nonumber \\
&& \quad 
+ \frac{1}{2\tau} (H^{m-\frac{3}{2}}_{h}, \nabla\times \bfE_h^{m})
- \frac{1}{2\tau} (H^{\frac{1}{2}}_{h}, \nabla\times \bfE_h^{0})
- \frac{1}{2\tau} (H^{\frac{3}{2}}_{h}, \nabla\times \bfE_h^{1}).
\label{re24}
\end{eqnarray}

Using \eqref{re24} and the inequality $ab\leq \frac{1}{4\delta}a^2 + \delta b^2$, we have
\begin{eqnarray}
Sta_4  &=& -\frac{\tau}{2\tau_0}\sum_{n=3}^{m}(\delta_{\tau}H^{n-2}_{h}+\delta_{\tau}H^{n-1}_{h}, \nabla\times \bfE_h^{n-1})
+ \frac{1}{2\tau_0} (H^{m-\frac{1}{2}}_{h}, \nabla\times \bfE_h^{m+1})  \nonumber \\
&& \quad 
+ \frac{1}{2\tau_0} (H^{m-\frac{3}{2}}_{h}, \nabla\times \bfE_h^{m})
- \frac{1}{2\tau_0} (H^{\frac{1}{2}}_{h}, \nabla\times \bfE_h^{0})
- \frac{1}{2\tau_0} (H^{\frac{3}{2}}_{h}, \nabla\times \bfE_h^{1}) \nonumber \\
&\leq& \frac{\tau}{4\tau_0}\sum_{n=3}^{m}(\frac{\mu_0}{2}||\delta_{\tau}H^{n-2}_{h}||^2
+\frac{\mu_0}{2}||\delta_{\tau}H^{n-1}_{h}||^2 + \frac{1}{\mu_0}||\nabla\times \bfE_h^{n-1}||^2)   \nonumber \\
&&\quad 
+ \frac{1}{8\mu_0}||\nabla\times \bfE_h^{m+1}||^2
+  \frac{\mu_0}{2\tau_0^2} ||H^{m-\frac{1}{2}}_{h}||^2  
+ \frac{1}{8\mu_0}||\nabla\times \bfE_h^{m}||^2
+  \frac{\mu_0}{2\tau_0^2} ||H^{m-\frac{3}{2}}_{h}||^2     \nonumber \\
&&\quad 
+ \frac{1}{8\mu_0}||\nabla\times \bfE_h^{0}||^2
+  \frac{\mu_0}{2\tau_0^2} ||H^{\frac{1}{2}}_{h}||^2  
+ \frac{1}{8\mu_0}||\nabla\times \bfE_h^{1}||^2
+  \frac{\mu_0}{2\tau_0^2} ||H^{\frac{3}{2}}_{h}||^2.
\label{re26}
\end{eqnarray}

By the same technique, we have
\begin{eqnarray}
Sta_5  &=& -\frac{\tau}{\mu_0}\sum_{n=3}^{m}(\delta_{2\tau}K_s^{n-1}, \nabla\times \bfE_h^{n-1})
+ \frac{1}{2\mu_0} (K_s^{m}, \nabla\times \bfE_h^{m+1})  \nonumber \\
&& \quad 
+ \frac{1}{2\mu_0} (K_s^{m-1}, \nabla\times \bfE_h^{m})
- \frac{1}{2\mu_0} (K_s^{1}, \nabla\times \bfE_h^{0})
- \frac{1}{2\mu_0} (K_s^{2}, \nabla\times \bfE_h^{1}) \nonumber \\
&\leq& \frac{\tau}{2\mu_0}\sum_{n=3}^{m}(||\delta_{2\tau}K_s^{n-1}||^2
+ ||\nabla\times \bfE_h^{n-1}||^2)   \nonumber \\
&&\quad 
+ \frac{1}{8\mu_0}||\nabla\times \bfE_h^{m+1}||^2
+  \frac{1}{2\mu_0} ||K_s^{m}||^2  
+ \frac{1}{8\mu_0}||\nabla\times \bfE_h^{m}||^2
+  \frac{1}{2\mu_0} ||K_s^{m-1}||^2     \nonumber \\
&&\quad 
+ \frac{1}{8\mu_0}||\nabla\times \bfE_h^{0}||^2
+  \frac{1}{2\mu_0} ||K_s^{1}||^2  
+ \frac{1}{8\mu_0}||\nabla\times \bfE_h^{1}||^2
+  \frac{1}{2\mu_0} ||K_s^{2}||^2.
\label{re28}
\end{eqnarray}

Substituting all the above estimates of $Sta_k$ into \eqref{re10} and collecting like terms, we obtain
\begin{eqnarray}
 && ENG_m \leq ENG_0 + \left[\frac{C_{\mathrm{in}}^2\tau^2h^{-2}}{2\mu_0} + 
\frac{\sigma_0C_{\mathrm{tr}}^2\tau^2h^{-2}}{2\tau_0}
+ \frac{\tau C_{\mathrm{in}}^2\tau^2h^{-2}}{4\mu_0\tau_0}
+ \frac{C_{\mathrm{in}}\tau^2h^{-1}}{8\mu_0\tau_0}\right] ||\delta_{\tau}\bfE_h^{m+ \frac{1}{2}}||^2
\nonumber \\
&&\quad + \tau (\frac{C_{\mathrm{in}}\tau h^{-1}}{4\mu_0\tau_0} + \frac{1}{2\mu_0})\sum_{n=1}^{m}||K_s^n||^2
+ \tau\cdot \frac{C_{\mathrm{in}}\tau h^{-1}}{4\mu_0\tau_0}\sum_{n=0}^{m-1}||\delta_{\tau}\bfE_h^{n+ \frac{1}{2}}||^2   \nonumber \\
&&\quad 
+ \tau (\frac{1}{2\mu_0} + \frac{1}{4\mu_0\tau_0} + \frac{1}{2\mu_0})\sum_{n=1}^{m}||\nabla\times\bfE_h^{n}||^2 +  \frac{\tau \mu_0}{2}||H_h^{m+\frac{1}{2}}||^2
+ \tau \mu_0\sum_{n=0}^{m-1}||H_h^{n+\frac{1}{2}}||^2    \nonumber \\
&&\quad 
+ \frac{\tau \mu_0}{4\tau_0}\sum_{n=1}^{m-1}||\delta_{\tau}H_h^{n}||^2  
+ \frac{\mu_0}{2\tau_0^2}(||H_h^{m-\frac{1}{2}}||^2 + ||H_h^{m-\frac{3}{2}}||^2)  
+ \frac{\mu_0}{2\tau_0^2}(||H_h^{\frac{1}{2}}||^2 + ||H_h^{\frac{3}{2}}||^2)  \nonumber \\
&&\quad 
+ \frac{1}{4\mu_0} (||\nabla\times\bfE_h^{m+1}||^2 + ||\nabla\times\bfE_h^{m}||^2)
+ \frac{1}{4\mu_0} (||\nabla\times\bfE_h^{1}||^2 + ||\nabla\times\bfE_h^{0}||^2) \nonumber \\
&&\quad + \frac{\tau}{2\mu_0} \sum_{n=2}^{m-1}||\delta_{2\tau}K_s^{n}||^2
+ \frac{1}{2\mu_0} (||K_s^{m}||^2 + ||K_s^{m-1}||^2 + ||K_s^{1}||^2  + ||K_s^{2}||^2).
\label{re30}
\end{eqnarray}

(IV)~ Now all the right hand side terms of \eqref{re30}, except $\frac{\tau \mu_0}{4\tau_0}\sum_{n=1}^{m-1}||\delta_{\tau}H_h^{n}||^2$ and $\frac{\mu_0}{2\tau_0^2}(||H_h^{m-\frac{1}{2}}||^2 + ||H_h^{m-\frac{3}{2}}||^2) $, can be controlled by choosing $\tau$ small enough and using the discrete Gronwall inequality. To bound these two terms,  squaring the following identity
\begin{eqnarray}
H_h^{m-\frac{1}{2}} &=& (H_h^{m-\frac{1}{2}}-H_h^{m-\frac{3}{2}})
+  (H_h^{m-\frac{3}{2}}-H_h^{m-\frac{5}{2}})
+ \cdots +  (H_h^{\frac{3}{2}}-H_h^{\frac{1}{2}}) +  H_h^{\frac{1}{2}}  \nonumber \\
&=& \tau  \sum_{n=1}^{m-1}\delta_{\tau}H_h^{n} + H_h^{\frac{1}{2}},
\label{re32}
\end{eqnarray}
we have
\begin{eqnarray}
||H_h^{m-\frac{1}{2}}||^2 &\leq& 2(\tau^2 || \sum_{n=1}^{m-1}\delta_{\tau}H_h^{n}||^2
 + ||H_h^{\frac{1}{2}}||^2)
\leq 2 ||H_h^{\frac{1}{2}}||^2 + 2\tau^2 (\sum_{n=1}^{m-1}1^2)( \sum_{n=1}^{m-1}||\delta_{\tau}H_h^{n}||^2 )  \nonumber \\
&\leq& 2 ||H_h^{\frac{1}{2}}||^2 + 2T\tau \sum_{n=1}^{m-1}||\delta_{\tau}H_h^{n}||^2.
\label{re34}
\end{eqnarray}

By the same argument, we have
\begin{eqnarray}
||H_h^{m-\frac{3}{2}}||^2 = ||\tau  \sum_{n=1}^{m-2}\delta_{\tau}H_h^{n} + H_h^{\frac{1}{2}}||^2
\leq 2 ||H_h^{\frac{1}{2}}||^2 + 2T\tau \sum_{n=1}^{m-2}||\delta_{\tau}H_h^{n}||^2.
\label{re35}
\end{eqnarray}

To bound $\delta_{\tau}H^{n}_{h}$, taking $\psi_h=\delta_{\tau}H^{n}_{h}$ in \eqref{sc2}, we obtain
\begin{eqnarray}
\mu_0 ||\delta_{\tau}H^{n}_{h}||^2 &=&   - (\nabla \times E_{h}^{n}, \delta_{\tau}H^{n}_{h})
 -(K_{s}^{n}, \delta_{\tau}H^{n}_{h}) \nonumber \\
&\leq& (\frac{\mu_0}{4}||\delta_{\tau}H^{n}_{h}||^2 + \frac{1}{\mu_0} ||\nabla \times E_{h}^{n}||^2)
+  (\frac{\mu_0}{4}||\delta_{\tau}H^{n}_{h}||^2 + \frac{1}{\mu_0} ||K_{s}^{n}||^2),
 \label{re36} 
\end{eqnarray}
which leads to
\begin{eqnarray}
\mu_0 ||\delta_{\tau}H^{n}_{h}||^2 
\leq \frac{2}{\mu_0} ||\nabla \times E_{h}^{n}||^2
+   \frac{2}{\mu_0} ||K_{s}^{n}||^2.
 \label{re38} 
\end{eqnarray}

We can use similar techniques 
to bound the terms $K_s^{m}$ and $K_s^{m-1}$ in \eqref{re30}, even though they are given source functions and we can keep them as they are. When $m$ is even, we have
\begin{eqnarray}
||K_s^{m}||^2 = ||2\tau  \sum_{n=1}^{m-1}\delta_{2\tau}K_s^{n} + K_s^{0}||^2
\leq 2 ||K_s^{0}||^2 + 4T\tau \sum_{n=1}^{m-1}||\delta_{2\tau}K_s^{n}||^2.
\label{re40}
\end{eqnarray}

When $m$ is odd, we have
\begin{eqnarray}
||K_s^{m}||^2 = ||2\tau  \sum_{n=2}^{m-1}\delta_{2\tau}K_s^{n} + K_s^{1}||^2
\leq 2 ||K_s^{1}||^2 + 4T\tau \sum_{n=2}^{m-1}||\delta_{2\tau}K_s^{n}||^2.
\label{re42}
\end{eqnarray}

Substituting the estimates of \eqref{re34}, \eqref{re35}, and \eqref{re38}-\eqref{re42}  into \eqref{re30}, and 
choosing $\tau$ small enough, e.g.,
\begin{eqnarray}
&& \tau\leq 1, ~~ \frac{C_{\mathrm{in}}^2\tau^2h^{-2}}{2\mu_0}\leq \frac{\eps_0}{8}~ (\mbox{or}~ \tau \leq \frac{h}{2C_{\mathrm{in}}C_v}), ~~\frac{\sigma_0C_{\mathrm{tr}}^2\tau^2h^{-2}}{2\tau_0}\leq \frac{\eps_0}{8} ~(\mbox{or}~ \tau\leq \frac{h\sqrt{\eps_0\tau_0}}{2C_{\mathrm{tr}}\sqrt{\sigma_0}}), \nonumber \\
&& \frac{C_{\mathrm{in}}^2\tau^2h^{-2}}{4\mu_0\tau_0}\leq \frac{\eps_0}{8}~ (\mbox{or}~ \tau\leq \frac{h\sqrt{\tau_0}}{\sqrt{2}C_{\mathrm{in}}C_v}), ~~\frac{C_{\mathrm{in}}\tau h^{-1}}{8\mu_0\tau_0}\leq \frac{\eps_0}{8} ~(\mbox{or}~ \tau\leq \frac{h\tau_0}{C_{\mathrm{in}}C_v}),
 \label{re44} 
\end{eqnarray}
which is equivalent to \eqref{CFL}, we complete the proof by the discrete Gronwall inequality.
\end{proof}
\subsection{The error estimate}

To carry out the error estimate of our scheme \eqref{sc1}-\eqref{sc2}, we split the solution error for $\bfE$ as follows:
\begin{eqnarray}
\bfE(\bfx,t_n) - \bfE_h^n =  (\bfE(\bfx,t_n) - \Pi_c\bfE^n) - (\bfE_h^n - \Pi_c\bfE^n) 
=: \bfE^n_{\xi} - \bfE^n_{\eta},
\label{re50} 
\end{eqnarray}
where for simplicity we denote $\bfE^n:=\bfE(\bfx,t_n)$. Similarly, we can define the solution error 
for $H$ as follows:
\begin{eqnarray}
H(\bfx,t_n) - H_h^n =  (H(\bfx,t_n) - \Pi_2H^n) - (H_h^n - \Pi_2H^n) 
=: H^n_{\xi} - H^n_{\eta}.
\label{re52} 
\end{eqnarray}

Integrating \eqref{p5} with respect to $t$ from $t_{n-1}$ to $t_{n+1}$, then divide the result by $2\tau_0\tau$, we obtain
\begin{eqnarray}
&&   \eps_0 (\delta_{2\tau}\pa_{t} \bfE^n, \bfphi) +  \frac{\eps_0}{\tau_0} (\delta_{2\tau}\bfE^n, \bfphi)
+  \frac{1}{\mu_0} (\frac{1}{2\tau}\int_{t_{n-1}}^{t_{n+1}}\nabla\times \bfE~dt, \nabla\times \bfphi)  \label{re54} \\
&=&   \frac{1}{\tau_0}(\frac{1}{2\tau}\int_{t_{n-1}}^{t_{n+1}}H~dt, \nabla\times \bfphi) -  \frac{1}{\mu_0} (\frac{1}{2\tau}\int_{t_{n-1}}^{t_{n+1}}K_s~dt, \nabla\times \bfphi)  - \frac{\sigma_0}{\tau_0} \langle \frac{1}{2\tau}\int_{t_{n-1}}^{t_{n+1}}\bfE~dt, \bfphi\rangle_{\Gamma}.
    \nonumber
   \end{eqnarray}

Integrating \eqref{p6} with respect to $t$ from $t_{n-\frac{1}{2}}$ to $t_{n+\frac{1}{2}}$, then divie the result by $\tau$, we have
\begin{eqnarray}
 \mu_0 (\delta_{\tau}H^{n}, \psi) = -(\frac{1}{\tau}\int_{t_{n-\frac{1}{2}}}^{t_{n+\frac{1}{2}}}\nabla\times \bfE~dt, \psi) -  (\frac{1}{\tau}\int_{t_{n-\frac{1}{2}}}^{t_{n+\frac{1}{2}}}K_s~dt, \psi). \label{re56} 
   \end{eqnarray}

Now subtracting \eqref{sc1} from \eqref{re54} with $\bfphi=\bfphi_h$, 
 \eqref{sc2} from \eqref{re56} with $\psi=\psi_h$, 
  using the error notation we introduced, and after some lengthy but straightforward algebra, we can obtain the error equations:
  \begin{eqnarray}
 && \eps_0 (\delta_{\tau}^2\bfE_{\eta}^{n}, \bfphi_h) 
 + \frac{\eps_0}{\tau_0}(\delta_{2\tau}\bfE_{\eta}^{n}, \bfphi_h) +\frac{1}{\mu_0} (\nabla \times \bfE^n_{\eta},\nabla \times \bfphi_h) \nonumber \\
&=&  \frac{1}{\tau_0}(\ov{H}^{n}_{\eta}, \nabla\times \bfphi_h) - \frac{\sigma_0}{\tau_0} \langle\bfE_{\eta}^n, \bfphi_h\rangle_{\Gamma} 
+\eps_0 (\delta_{\tau}^2\bfE_{\xi}^{n}, \bfphi_h) 
+ \eps_0 (\delta_{2\tau}\pa_t\bfE^{n} - \delta_{\tau}^2\bfE^{n}, \bfphi_h)  \nonumber \\
&& \quad
 + \frac{\eps_0}{\tau_0}(\delta_{2\tau}\bfE_{\xi}^{n}, \bfphi_h) 
+ \frac{1}{\mu_0} (\frac{1}{2\tau}\int_{t_{n-1}}^{t_{n+1}}\nabla \times (\bfE - \Pi_c\bfE^n)~dt,\nabla \times \bfphi_h)  \nonumber \\
&&\quad 
+  \frac{1}{\tau_0} (\frac{1}{2\tau}\int_{t_{n-1}}^{t_{n+1}}(H - \Pi_2\ov{H}^n)~dt,\nabla \times \bfphi_h) 
+  \frac{\sigma_0}{\tau_0} \langle \frac{1}{2\tau}\int_{t_{n-1}}^{t_{n+1}}(\bfE - \Pi_c\bfE^n)~dt, \bfphi_h\rangle_{\Gamma}    \nonumber \\
&&\quad 
-  \frac{1}{\mu_0} (\frac{1}{2\tau}\int_{t_{n-1}}^{t_{n+1}}(K_s - K_s^n)~dt,\nabla \times \bfphi_h), 
 \label{re58} 
\end{eqnarray}
and
 \begin{eqnarray}
 && \mu_0 (\delta_{\tau}H^{n}_{\eta}, \psi_h) + (\nabla \times \bfE_{\eta}^{n}, \psi_h) 
= \mu_0 (\delta_{\tau}H^{n}_{\xi}, \psi_h)  \nonumber \\
&&\quad
- (\frac{1}{\tau}\int_{t_{n-\frac{1}{2}}}^{t_{n+\frac{1}{2}}}\nabla\times (\Pi_c\bfE^n - \bfE)~dt, \psi_h)
-(\frac{1}{\tau}\int_{t_{n-\frac{1}{2}}}^{t_{n+\frac{1}{2}}}(K_{s}^{n}-K_s)~dt, \psi_h). \label{re60} 
\end{eqnarray}

Note that the error equation \eqref{re58} has the first five terms in the same form as  scheme \eqref{sc1}, and the rest terms are extra error terms due to the time approximation, and projection or interpolation.
Furthermore, the error equation \eqref{re60} has the first two terms in the same form as  scheme \eqref{sc2}, and the rest three terms are  error terms due to the time approximation, and projection or interpolation.
We want to remark that those extra error terms are $O(\tau^2 + h^r)$ by the interpolation error estimate \eqref{Int} and the projection error estimate \eqref{Pro}. Following the stability proof (due to its technicality, we skip the proof details), we have the following error estimate for our scheme \eqref{sc1}-\eqref{sc2}:
\begin{eqnarray}
||\nabla\times (\bfE^{m}_h - \bfE^{m})|| + ||H_h^{m+\frac{1}{2}} - H^{m+\frac{1}{2}} || \leq C(\tau^2 + h^r), 
\label{re62} 
\end{eqnarray}
where $r$ is the degree of our finite element spaces $\bfV_h$ and $U_h$.

\section{Numerical results}

In this section, we present some numerical results solved by our proposed numerical scheme.
The first example is developed to test the convergence of our scheme by using FEniCS \cite{Logg}, and the rest are carried out using NGSolve \cite{Schoeberl2021} to demonstrate that our reformulated graphene model can still generate surface plasmon polaritons.

  \subsection{Test of convergence rates}
The first example is developed to test the  convergence rate of our numerical scheme by a manufactured exact solution given as follows (adapted from our previous work \cite{Li_CMAWA2023}):
\begin{eqnarray*}
&& \mathbf{E}(x,y,t) =\begin{pmatrix}
E_x \\
E_y \\
\end{pmatrix} =\begin{pmatrix}
\sin(2 \pi x)\sin(2\pi y)\sin(2\pi t)\\
\cos(2 \pi x)\cos(2\pi y)\sin(2\pi t) \\
\end{pmatrix}, \\
&& 
H_1(x,y,t) = \frac{1}{1+4\pi^2}\sin(2\pi x)\sin(2\pi y)\sin(2\pi t), \\
&&
H_2(x,y,t) = \frac{1}{1+4\pi^2}\sin(2\pi x)\sin(2\pi y)(2\pi \cos(2\pi t) -2\pi \exp(-t)),
\end{eqnarray*}
which satisfies the following weak form for the graphene model equations: For any $\bfphi \in H_0(curl;\Omega)$ and $\psi \in L^2(\Omega)$,
\begin{eqnarray}
    && \tau_0\eps_0 (\pa_{tt} \bfE, \bfphi) + \eps_0 (\pa_t \bfE, \bfphi) + \frac{\tau_0}{\mu_0} ( \nabla \times \bfE , \nabla \times \bfphi) = (H_1 ,\nabla \times \bfphi) - \sigma_0 \langle \bfE, \bfphi \rangle_{\Gamma}  \nonumber \\
    && \hskip 1.5in   + (\bff_1, \bfphi)+ (\tau_0\pa_t \bff_1, \bfphi) + \frac{\tau_0}{\mu_0} (f_2, \nabla 
    \times \bfphi), \quad\mbox{in } \O_1, \label{ex1} \\
    && \mu_0(\pa_t H_1,\psi) = (-\nabla \times \bfE + f_2,\psi), \quad\mbox{in } \O_1, \label{ex2} \\
&& \tau_0\eps_0 (\pa_{tt} \bfE, \bfphi) + \eps_0 (\pa_t \bfE, \bfphi) + \frac{\tau_0}{\mu_0} ( \nabla \times \bfE , \nabla \times \bfphi) = (H_2 ,\nabla \times \bfphi) - \sigma_0 \langle \bfE, \bfphi\rangle_{\Gamma}  \nonumber \\
    && \hskip 1.5in + (\bff_3, \bfphi) + (\tau_0\pa_t \bff_3, \bfphi) + \frac{\tau_0}{\mu_0} (f_4, \nabla 
    \times \bfphi), \quad\mbox{in } \O_2, \label{ex3} \\
    && \mu_0(\pa_t H_2, \psi)= (-\nabla \times \bfE + f_4,\psi), \quad\mbox{in } \O_2. \label{ex4}
\end{eqnarray}
Here the extra source terms $\bff_1, f_2, \bff_3, f_4$ can be derived from the given exact solution $\bfE, \bfJ, H_1$, and $H_2$ as in \cite{Li_CMAWA2023}. Note that the weak form is derived from \cite[(4.1)-(4.5)]{Li_CMAWA2023} by following the same steps to get our weak formulation \eqref{p5}-\eqref{p6} with added source terms.

 For simplicity, we choose the physical domain $\O = (0,1)^2$, which is split into two subdomains $\O_1 = (0,1)\times (0.5,1)$ and $\O_2 = (0,1)\times (0,0.5)$ with interface $\Gamma=\{y=0.5, x\in [0,1] \}$.
 We apply our developed scheme \eqref{sc1}-\eqref{sc2} to solve \eqref{ex1}-\eqref{ex4} 
 with physical parameters $\epsilon_0 = \mu_0 = \tau_0= \sigma_0=1$.

Here the added source terms $\bff_1, f_2$, and $\bff_3$ can be calculated from the given exact solution $\bfE, H_1$, and $H_2$. We use our scheme (\ref{sc1})-(\ref{sc2}) on the 
the same parameters and physical domain setup as \cite{Li_CMAWA2023}. 
We solve this example with a fixed small time step size $\tau=1\times 10^{-4}$, and various mesh sizes for $N_t=1000$ time steps. The obtained $L^2$ errors are presented in Table 1 and Table 2 for the RTN finite element spaces $U_h$ and $\bfV_h$ with polynomial degree $r=1, 2$, respectively.

\begin{table}[htbp]
	\begin{center}
		\caption{The obtained errors for  $ N_t=1000, \tau = 1\times 10^{-4}, r=1$. }
		{\begin{tabular}{c|cccccc}
				\hline
h& $\|\bfE -\bfE_h \|_{L^2(\O)}$ &rate&  $\| H -H_h \|_{L^2(\O)}$ & rate  \\ \hline

1/4 & $1.995472\times 10^{-2}$ & & $4.037259 \times 10^{-4}$ \\  
1/8 & $1.005105\times 10^{-3}$ & 0.989385  & $2.044797\times 10^{-4}$ & 0.981418 \\
1/16 & $5.035177\times 10^{-3}$ & 0.997231 & $1.014174\times 10^{-4}$ & 1.011653 \\
1/32 & $2.518828\times 10^{-3}$ & 0.999290 & $4.906353\times 10^{-5}$  & 1.047582 \\
1/64 & $1.259572\times 10^{-3}$ &0.999819& $2.327676\times 10^{-5}$  & 1.075761 \\
1/128 & $6.199001\times 10^{-4}$ & 1.022826 & $1.082733\times 10^{-5}$  & 1.104213\\
\hline
		\end{tabular}}
		\label{tab: Convective}
	\end{center}
\end{table}

\begin{table}[htbp]
	\begin{center}
		\caption{The obtained errors for  $ N_t=1000, \tau = 1\times 10^{-4}, r=2$. }
		{\begin{tabular}{c|cccccc}
				\hline
h& $\|\bfE -\bfE_h \|_{L^2(\O)}$ &rate&  $\| H -H_h \|_{L^2(\O)}$ & rate  \\ \hline

1/4 & $5.028653\times 10^{-2}$ & & $2.172637 \times 10^{-4}$ \\  
1/8 & $1.290415\times 10^{-3}$ & 1.962337  & $1.035932\times 10^{-4}$ & 1.068517 \\
1/16 & $3.186606\times 10^{-3}$ & 2.017743 & $4.847654\times 10^{-4}$ & 1.095571 \\
1/32 & $7.545130\times 10^{-3}$ & 2.078403 & $1.887592\times 10^{-5}$  & 1.360740 \\
1/64 & $2.112674\times 10^{-3}$ &1.836476& $4.771321\times 10^{-6}$  & 1.984086 \\
1/128 & $5.979253\times 10^{-4}$ &  1.821033& $1.182453\times 10^{-6}$  & 2.012605\\
\hline
		\end{tabular}}
		\label{tab: Convective}
	\end{center}
\end{table}

Then we test the convergence rate for $\tau$ by fixing $\tau=\frac{h}{200}$ to satisfy the stability condition. In Tables 3-4 the obtained $L^2$ errors for r=1,2 are presented.

\begin{table}[htbp]
	\begin{center}
		\caption{The obtained errors for  $ T=0.01, \tau = \frac{h}{200}, r=1$. }
		{\begin{tabular}{c|cccccc}
				\hline
h& $\|\bfE -\bfE_h \|_{L^2(\O)}$ &rate&  $\| H -H_h \|_{L^2(\O)}$ & rate  \\ \hline

1/10 & $8.441669\times 10^{-3}$ & & $1.768658\times 10^{-4}$ \\  
1/20 & $4.125569\times 10^{-3}$ & 1.032935  & $8.359509\times 10^{-5}$ & 1.081165 \\
1/40 & $2.038308\times 10^{-3}$ & 1.017221 & $3.914783\times 10^{-5}$ & 1.094486 \\
1/80 & $1.012953\times 10^{-3}$ & 1.008805 & $1.868672\times 10^{-5}$  & 1.066919 \\
1/160 & $5.069774\times 10^{-4}$ &0.998574& $1.035698\times 10^{-5}$  & 0.851410 \\
\hline
		\end{tabular}}
		\label{tab: Convective}
	\end{center}
\end{table}

\begin{table}[htbp]
	\begin{center}
		\caption{The obtained errors for  $ T=0.01, \tau = \frac{h}{200}, r=2$. }
		{\begin{tabular}{c|cccccc}
				\hline
h& $\|\bfE -\bfE_h \|_{L^2(\O)}$ &rate&  $\| H -H_h \|_{L^2(\O)}$ & rate  \\ \hline

1/10 & $5.648216\times 10^{-3}$ & & $2.931658\times 10^{-4}$ \\  
1/20 & $1.368327\times 10^{-3}$ & 2.045382  & $1.216507\times 10^{-4}$ & 1.268972 \\
1/40 & $3.277738\times 10^{-4}$ & 2.061641 & $5.237098\times 10^{-5}$ & 1.215905\\
1/80 & $7.647986\times 10^{-5}$ &2.099549 & $1.938811\times 10^{-5}$  & 1.433595 \\
1/160 & $2.140379\times 10^{-5}$ &1.837214& $4.807332\times 10^{-6}$  & 2.011864 \\
\hline
		\end{tabular}}
		\label{tab: Convective}
	\end{center}
\end{table}

\subsection{Simulation of surface plasmon polaritons on the graphene sheets}
In this section, we provide several numerical examples to illustrate the effectiveness of our graphene model in simulating the propagation of surface plasmon polaritons (SPPs) on graphene sheets.

To effectively demonstrate the SPPs propagating on  graphene sheets, we use a perfectly matched layer (PML) to surround the physical domain $\O$.  Since our current graphene model only involves the electric and magnetic fields, we just adopt the 2-D TEz Berenger's PML model \cite{Berenger1994}, which can be written as (cf. \cite[(12)-(15)]{Li2012cloak}): For any $(\bfx,t) \in \O_{\mathrm{pml}}\times (0,T]$,
\begin{eqnarray}
    && \eps_0 \pa_t E_x  + \sigma_y E_x = \pa_y (H_{zx} + H_{zy}),   \label{z1} \\
    && \eps_0 \pa_t E_y + \sigma_x E_y = -\pa_x (H_{zx} + H_{zy}),   \label{z1} \\
    && \mu_0 \pa_t H_{zx} + \frac{\mu_0}{\eps_0} \sigma_x H_{zx} = -\pa_x E_y,  \label{z3} \\
    && \mu_0 \pa_t H_{zy} + \frac{\mu_0}{\eps_0} \sigma_y H_{zy} = \pa_y E_x, \label{z4} 
\end{eqnarray}
where $\sigma_x(x)$ and $\sigma_y(y)$ are the nonnegative damping functions in the $x$ and $y$ directions, respectively, and $\O_{\mathrm{pml}}$ represents the PML region. Here 
$H_{zx}$ and $H_{zy}$ are the two splitted components of the orginal magnetic field $H_z$, i.e., $H_z:=H_{zx}
+H_{zy}$.

We propose the following
finite element scheme for the above PML model in $\O_{\mathrm{pml}}$: For any $n \geq 0$, find $\bfE_h^{n+1} \in \bfV_h^0, H_{zx}^{n+\frac{1}{2}}, H_{zy}^{n+\frac{1}{2}}  \in U_h$ such that 
\begin{flalign}
 & \eps_0 (\delta_{2\tau}\bfE_h^{n}, \bfphi_h) 
 + \eps_0(D_1\bfE_h^{n}, \bfphi_h) = (\ov{H}^{n}_{zx, h}, \nabla\times \bfphi_h) +(\ov{H}^{n}_{zy, h}, \nabla\times \bfphi_h), \label{pm1} \\
 & \mu_0 (\delta_{\tau}H^{n}_{zx,h}, \psi_h) + \frac{\mu_0}{\eps_0}(\sigma_x\ov{H}_{zx,h}^{n}, \psi_h)  =  - (\pa_x E_{y,h}^{n}, \psi_h), \label{pm2} \\
  & \mu_0 (\delta_{\tau}H^{n}_{zy,h}, \varphi_h ) + \frac{\mu_0}{\eps_0}(\sigma_y\ov{H}_{zy,h}^{n}, \varphi_h )  = (\pa_x E_{x,h}^{n}, \varphi_h ), \label{pm3}
\end{flalign}
hold true  for any test functions $\bfphi_h \in \bfV_h^0$, $\psi_h$, $\varphi_h \in U_h$, where $D_1 = \mathrm{diag}(\sigma_y, \sigma_x)$.

To simplify the implementation, we merge the graphene scheme (\ref{sc1})-(\ref{sc2}) and the PML scheme (\ref{pm1})-(\ref{pm3}) together by using subdomain dependent coefficients and 
rewrite them as follows:

\begin{flalign}
 & \left((\frac{\eps_0}{\tau^2}I+\frac{D_1}{2\tau}+\frac{C_1\eps_0}{2\tau\tau_0}I)\bfE^{n+1}_h, \bfphi_h\right)
 =\left(\frac{2\eps_0}{\tau^2}I\bfE^n_h, \bfphi_h\right)-\left((\frac{\eps_0}{\tau^2}I-\frac{D_1}{2\tau}-\frac{C_1\eps_0}{2\tau\tau_0}I)\bfE^{n+1}_h, \bfphi_h\right) \nonumber\\
 & \hskip 2in -\frac{C_1}{\mu_0}( \nabla \times \bfE^n_h, \nabla\times \bfphi_h) -\frac{C_1}{\mu_0}( K_{sh}^n, \nabla\times \bfphi_h) - \frac{\sigma_0}{\tau_0}\langle\bfE^n_h, \bfphi_h\rangle_{\Gamma}  \nonumber \\
& \hskip 2in +\frac{C_1}{2\tau_0} (H_{zx,h}^{n+\frac{1}{2}}+H_{zx,h}^{n-\frac{1}{2}}, \nabla \times \bfphi_h) +  \frac{C_1}{2\tau_0} (H_{zy,h}^{n+\frac{1}{2}}+H_{zy,h}^{n-\frac{1}{2}}, \nabla \times \bfphi_h) \nonumber \\
& \hskip 2in +\frac{C_2}{\tau} (H_{zx,h}^{n+\frac{1}{2}}-H_{zx,h}^{n-\frac{1}{2}}, \nabla \times \bfphi_h) +  \frac{C_2}{\tau} (H_{zy,h}^{n+\frac{1}{2}}-H_{zy,h}^{n-\frac{1}{2}}, \nabla \times \bfphi_h) ,   \label{N1}\\
&\left((\frac{\mu_0}{\tau}+\frac{\mu_0\sigma_x}{2\eps_0})H^{n+\frac{1}{2}}_{zx,h}, \psi_h\right)
 =\left((\frac{\mu_0}{\tau}-\frac{\mu_0\sigma_x}{2\eps_0})H^{n-\frac{1}{2}}_{zx,h}, \psi_h\right)
  - (\pa_x E_{y,h}^{n}, \psi_h)  -\frac{1}{2}(K_{s}^{n}, \psi_h ), \label{N2}\\
 & \left((\frac{\mu_0}{\tau}+\frac{\mu_0\sigma_y}{2\eps_0})H^{n+\frac{1}{2}}_{zy,h}, \psi_h\right)
 =\left((\frac{\mu_0}{\tau}-\frac{\mu_0\sigma_x}{2\eps_0})H^{n-\frac{1}{2}}_{zy,h}, \psi_h\right) + (\pa_y E_{x,h}^{n}, \psi_h)  -\frac{1}{2}(K_{s}^{n}, \psi_h ), \label{N3}
\end{flalign}
where we denote the identity matrix $I=\mathrm{diag}(1,1)$, write $ H_{h} = H_{zx,h} + H_{zy, h}$, and use the subdomain identity functions
\begin{eqnarray}
C_1 = \begin{cases}
1,  \quad\mbox{if}~ \bfx \in \Omega \\
0,  \quad\mbox{if}~ \bfx \in \Omega_{\mathrm{pml}}
\end{cases}, C_2 = \begin{cases}
0,  \quad\mbox{if}~ \bfx \in \Omega \\
1,  \quad\mbox{if}~ \bfx \in \Omega_{\mathrm{pml}}
\end{cases}.
\end{eqnarray}
 
 The damping functions $\sigma_x$ and $\sigma_y$ for the PML are chosen as a fourth-order polynomial:
$$ \sigma_x(x) = \begin{cases}
\sigma_{\max}(\frac{x-b_x}{dd})^4, & \mbox{when}~ x \geq b_x, \\
\sigma_{\max}(\frac{x - a_x}{dd})^4, & x \leq a_x, \\
0, & \text{elsewhere},
\end{cases} $$ 
where the coefficient $\sigma_{\max} = -\log(\mathrm{err}) \cdot 5 / (2\cdot \mathrm{dd} \cdot \eta)$ 
with $\mathrm{err}=10^{-7}$, $\eta = 377$ is the impedance of free space, and $dd$ denotes the thickness of the PML in the $x$ direction. 
The function $\sigma_y $ has the same form but varies with respect to the $y$ variable. 

In the rest of our simulations, we choose a physical domain $\O=[a_x, b_x]~\mu \mathrm{m}\times [a_y, b_y]~\mu \mathrm{m}$,
which is partitioned by a regular unstructured triangular mesh and is surrounded by the Berenger's PML with  thickness $12h_x$ and $12h_y$ in
the $x$ and $y$ directions, respectively. We denote $h_x$ and $h_y$ for the maximum mesh sizes in the $x$ and $y$ directions. The physical domain for the first three examples is a rectangle domain $\Omega = [-30, 30]~\mu \mathrm{m} \times [-10, 10]~\mu \mathrm{m}$ with $h_x=0.6~\mu \mathrm{m}, h_y=0.2~\mu \mathrm{m}$, 
and the domain for the last two examples is $\Omega = [-20, 20]~\mu \mathrm{m} \times [-20, 20]~\mu \mathrm{m}$ with $h_x=h_y=0.1~\mu \mathrm{m}$. In our simulations, we choose the relaxation time $\tau_0=1.2$ picoseconds (ps), the surface conductivity $\sigma_0$ given as:
\begin{eqnarray}
    && \sigma_0 = -\frac{q^2k_BT\tau}{\pi\hbar^2}\left(\frac{\mu_c}{k_BT}+2\ln(\exp(-\frac{\mu_c}{k_BT})+1)\right), \nonumber 
\end{eqnarray}
where we denote the electron charge $q=1.6022 \times e^{-19}\mathrm{C}$, the temperature $T=300~\mathrm{K}$, the reduced Plank constant $\hbar=1.0546 \times e^{-34} \mathrm{J} \cdot \mathrm{s}$, the Boltzmann constant $k_B=1.3806 \times e^{-23} \mathrm{J} \cdot \mathrm{K}^{-1}$, and the chemical potential $\mu_c=1.5~eV$ for Examples 1-4 and $\mu_c=0.8~eV$ for Example 5.

We use the time step $\tau = 8.3 \times 10^{-17}~s$, and run the  simulation for 10000 time steps for Example 3, 20000 time steps for Examples 1, 2, 4, and 100000 time steps for Example 5.

\subsubsection{Example 1. Bifurcated graphene sheets}

In this example, we firstly reproduce the results in \cite[Example 4]{Li_CMAWA2023} with our new scheme. The same setup with exemplary coarse mesh is shown in Figure \ref{bifurcate_straight_mesh}. 
\begin{figure}[H]
  \begin{center}
      \includegraphics[width=0.8\textwidth]{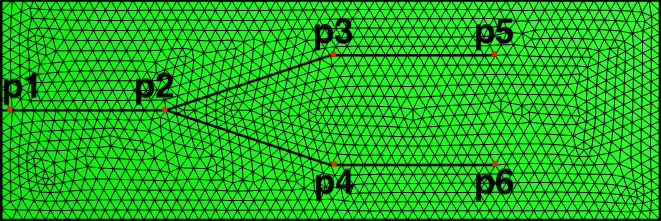}
 \caption{A bifurcated graphene sheet buried in vacuum with an exemplary coarse mesh. The structure is constructed by segments $\mathrm{\mathrm{p}}_1\mathrm{p}_2, \mathrm{p}_2\mathrm{p}_3, \mathrm{p}_3\mathrm{p}_4, \mathrm{p}_2\mathrm{p}_5, \mathrm{p}_5\mathrm{p}_6$. The coordinates for the marked points on the graph are correspondingly $\mathrm{p}_1(-30, 0) \mu \mathrm{m}, \mathrm{p}_2(-15, 0) \mu \mathrm{m}, \mathrm{p}_3(0, 5)\mu \mathrm{m}, \mathrm{p}_4(15, 5) \mu \mathrm{m}, \mathrm{p}_5(0, -5) \mu \mathrm{m}, \mathrm{p}_6(15, -5)\mu \mathrm{m}$.}
 \label{bifurcate_straight_mesh}
  \end{center}
\end{figure}

Snapshots of the magnetic field Hz are presented in Figure \ref{bifurcate_curve}, which clearly shows that a surface wave propagates along the graphene sheet same as \cite[Figure 8]{Li_CMAWA2023}.
To show the efficiency of our new algorithm, we made a comparison for the computational time between our new algorithm and the old algorithm proposed in our previous work \cite{Li_CMAWA2023}. The tests are carried out using NGSolve on a Mac mini with an Apple M2 chip. Under the same setup, our new algorithm takes 10,327.50 seconds, while the old algorithm needs 19,114.09 seconds. The new algorithm demonstrates a significant improvement in the computational time.

\begin{figure}[H]
 \begin{center}
    \begin{minipage}[b]{\textwidth}
    \includegraphics[width=0.5\textwidth]{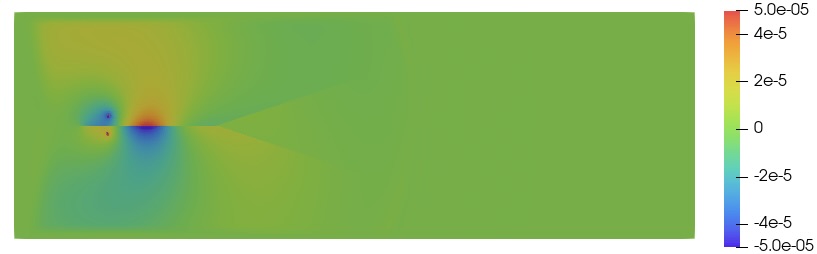}
    \includegraphics[width=0.5\textwidth]{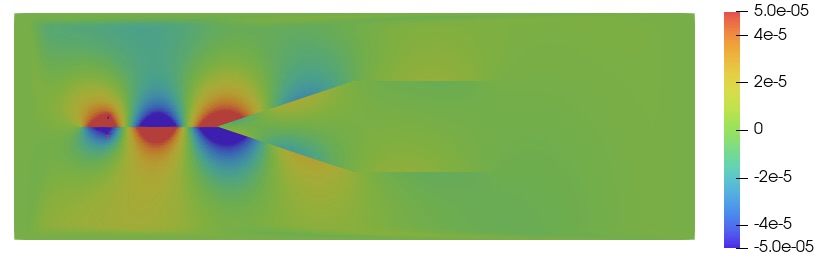}
    \includegraphics[width=0.5\textwidth]{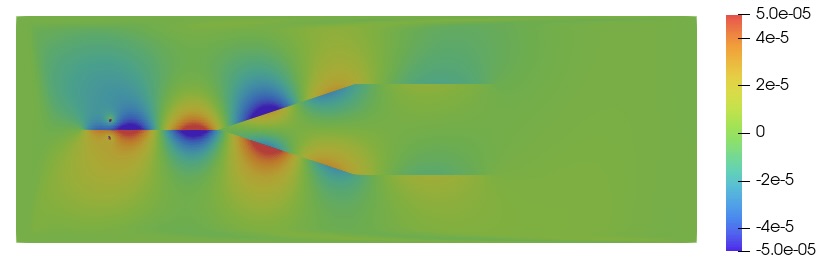}
    \includegraphics[width=0.5\textwidth]{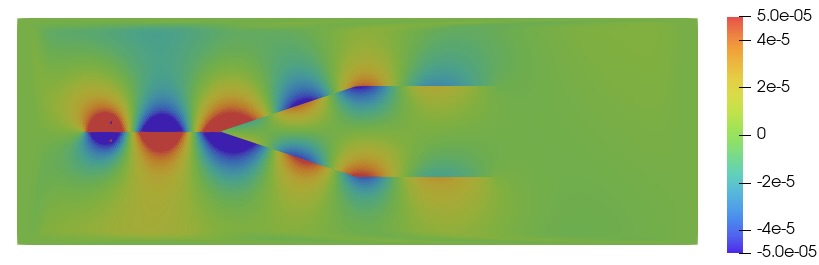}
    \includegraphics[width=0.5\textwidth]{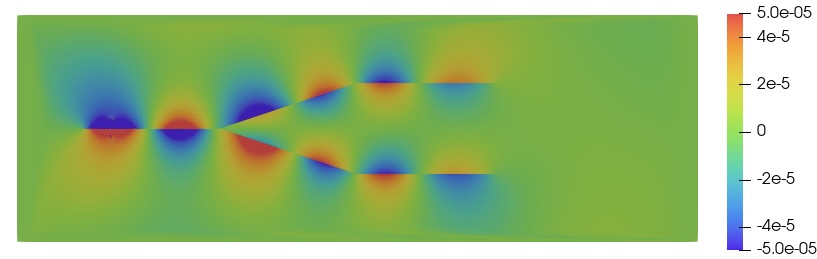}
    \includegraphics[width=0.5\textwidth]{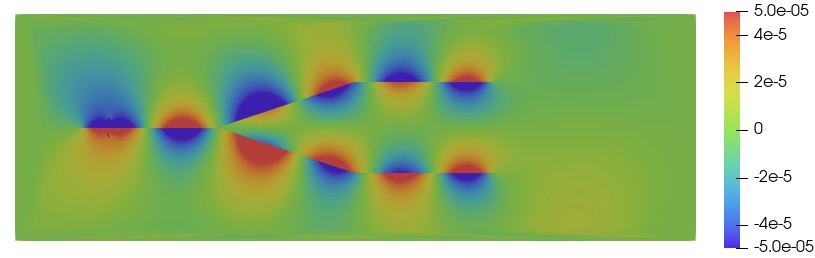}
  \end{minipage}
\hfill
  \caption{Snapshots of the contour plots of $H_z$ obtained at different time steps: 1000 (top left), 3000 (top right), 4000 (middle left), 5500 (middle right), 6000 (bottom left), and 10000 (bottom right).}
    \label{bifurcate_curve}
\end{center}
\end{figure}

Then we carry out a numerical simulation of the SPP propagation on a bifurcated curly graphene sheet. The setup with an exemplary coarse mesh is shown in Figure \ref{bifurcate_curve_mesh}. 

\begin{figure}[H]
  \begin{center}
      \includegraphics[width=0.8\textwidth]{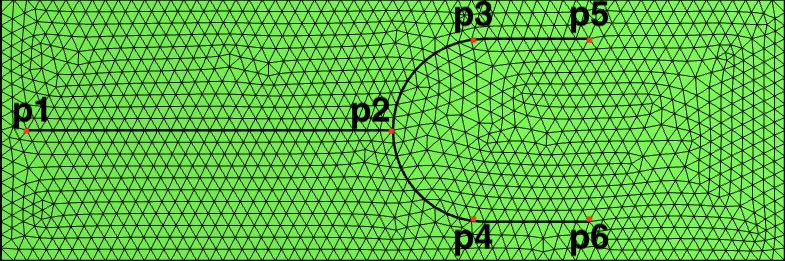}
 \caption{A bifurcated graphene sheet buried in vacuum with an exemplary coarse mesh. The structure is constructed by segments $\mathrm{p}_1\mathrm{p}_2, \mathrm{p}_3\mathrm{p}_5, \mathrm{p}_4\mathrm{p}_6$, and a semicircle centered at $(7, 0)\mu \mathrm{m}$ with radius $r=7 \mu \mathrm{m}$. The coordinates for the marked points on the graph are correspondingly $\mathrm{p}_1(-28, 0) \mu \mathrm{m}, \mathrm{p}_2(0, 0) \mu \mathrm{m}, \mathrm{p}_3(7, 7)\mu \mathrm{m}, \mathrm{p}_4(7, -7) \mu \mathrm{m}, \mathrm{p}_5(15, 7) \mu \mathrm{m}, \mathrm{p}_6(15, -7)\mu \mathrm{m}$.}
 \label{bifurcate_curve_mesh}
  \end{center}
\end{figure}

A pair of dipole source waves is placed at $(-27 \mu \mathrm{m}, 1 \mu \mathrm{m}) \text{ and } (-27 \mu \mathrm{m}, -1 \mu \mathrm{m})$ and imposed as $K_s=\pm \sin(2\pi f_0t)/h$ with $f_0=10~THz$ and $h=h_y$.
Snapshots of the magnetic field $H_z$  are presented in Figure \ref{Fig3}, which clearly show that 
 a surface wave propagates along the graphene sheet. 

\begin{figure}[H]
 \begin{center}
    \begin{minipage}[b]{\textwidth}
    \includegraphics[width=0.5\textwidth]{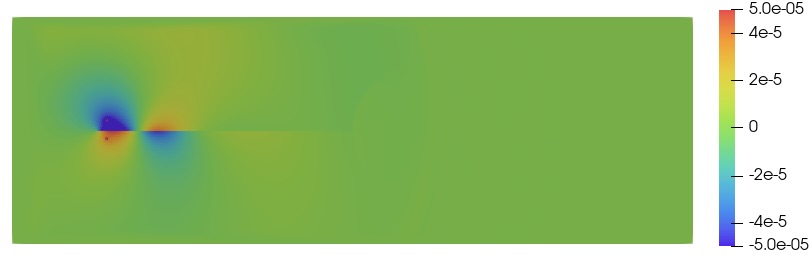}
    \includegraphics[width=0.5\textwidth]{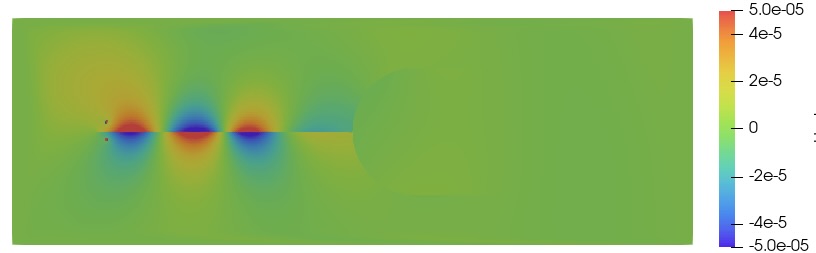}
    \includegraphics[width=0.5\textwidth]{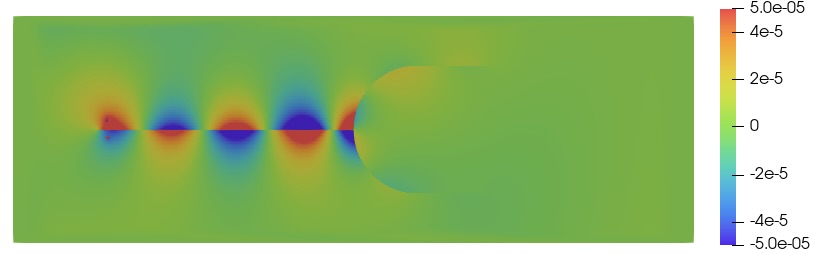}
    \includegraphics[width=0.5\textwidth]{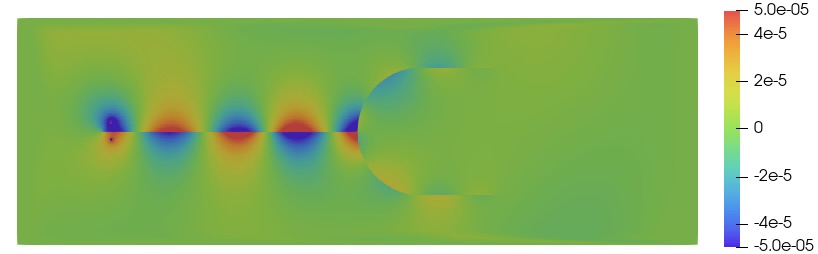}
    \includegraphics[width=0.5\textwidth]{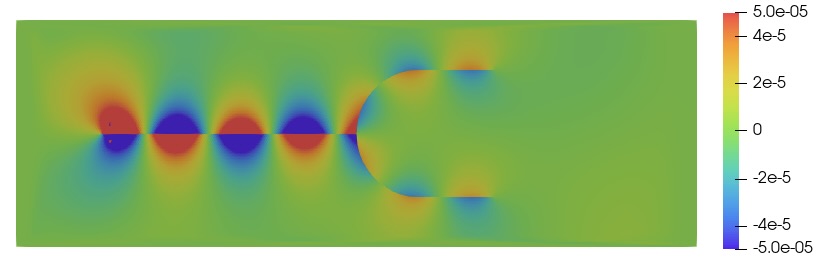}
    \includegraphics[width=0.5\textwidth]{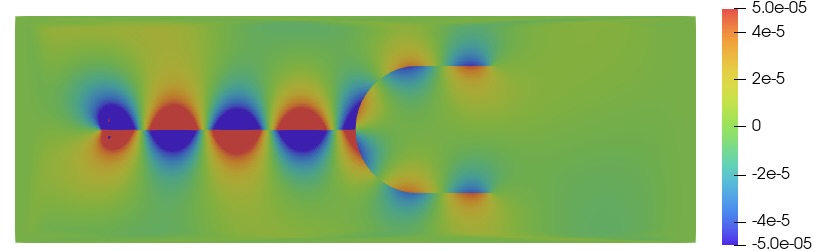}
  \end{minipage}
\hfill
 \end{center}
  \caption{Snapshots of the contour plots of $H_z$ obtained at different time steps: 1000 (top left), 3000 (top right), 5200 (middle left), 8000 (middle right), 10000 (bottom left), and 18800 (bottom right).}
\label{Fig3}  
\end{figure}

\subsubsection{Example 2. Four adjacent curved graphene sheets}
In this example, we present a numerical simulation of SPPs propagating along four separate
but adjacent  graphene sheets by our scheme. The simulation setup is shown in Figure \ref{adjacent_mesh}, 
 where four adjacent curved graphene sheets are embedded inside the physical domain $\O$. 
A pair of dipole source wave is placed at $(-18 \mu \mathrm{m}, 3.5 \mu \mathrm{m}) \text{ and } (-18 \mu \mathrm{m}, 2.5 \mu \mathrm{m})$. The other setup is the same as Example 1.  

\begin{figure}[H]
  \begin{center}      \includegraphics[width=0.8\textwidth]{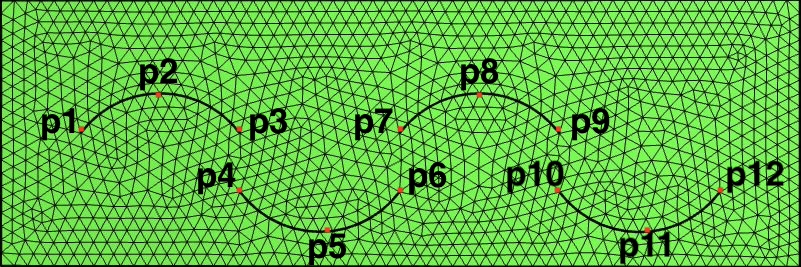}
 \caption{The setup of four adjacent curved graphene sheets and a sample coarse mesh. The structure is constructed by 4 circular arcs. The coordinates for the marked points on the graph are correspondingly $\mathrm{p}_1(-24, 0)\mu \mathrm{m}, \mathrm{p}_2(-18, 3)\mu \mathrm{m}, \mathrm{p}_3(-12, 0)\mu \mathrm{m}, \mathrm{p}_4(-12, -2h_y)\mu \mathrm{m}, \mathrm{p}_5(-6, -3-2h_y)\mu \mathrm{m},\\ \mathrm{p}_6(0, -2h_y)\mu \mathrm{m}, \mathrm{p}_7(0, 0)\mu \mathrm{m}, \mathrm{p}_8(6, 3)\mu \mathrm{m}, \mathrm{p}_9(12, 0)\mu \mathrm{m}, \mathrm{p}_{10}(12, -2h_y)\mu \mathrm{m}, \mathrm{p}_{11}(-18, -3-2h_y)\mu \mathrm{m},\\ \mathrm{p}_{12}(-24, -2h_y)\mu \mathrm{m}$.}
\label{adjacent_mesh}
\end{center}
\end{figure}

Snapshots of the magnetic field $H_z$  are presented in Figure \ref{adjacent}, which show that 
the surface wave propagates along the graphene sheets even though they are separated.  

\begin{figure}[H]
 \begin{center}
    \begin{minipage}[b]{\textwidth}
    \includegraphics[width=0.5\textwidth]{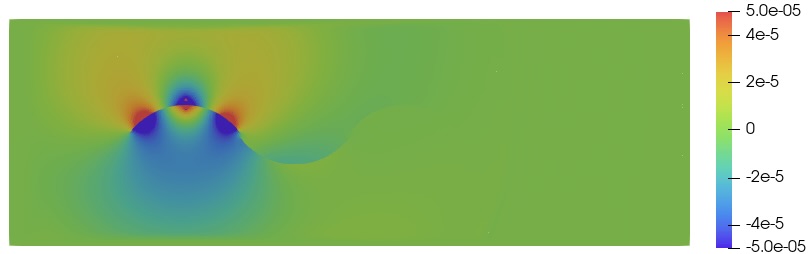}
    \includegraphics[width=0.5\textwidth]{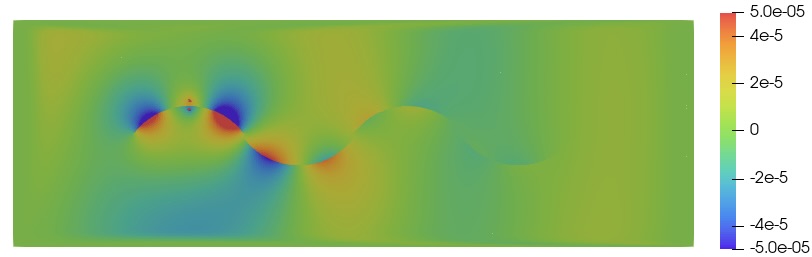}
    \includegraphics[width=0.5\textwidth]{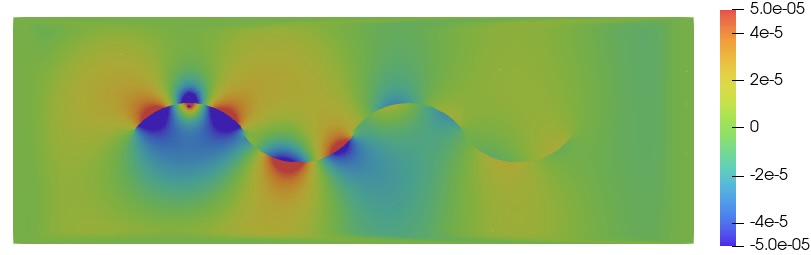}
    \includegraphics[width=0.5\textwidth]{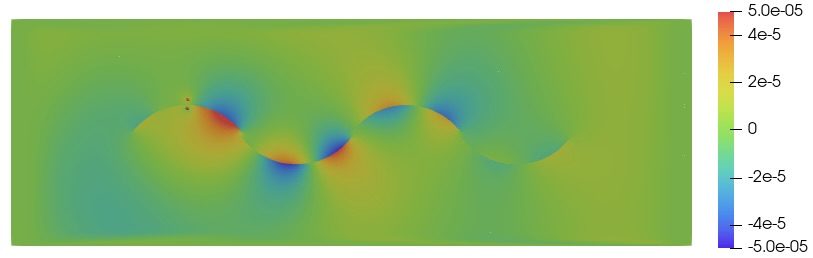}
    \includegraphics[width=0.5\textwidth]{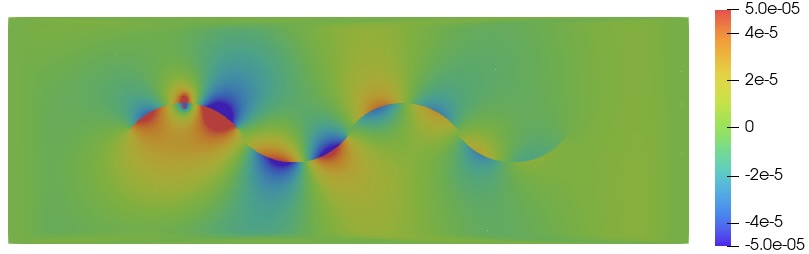}
    \includegraphics[width=0.5\textwidth]{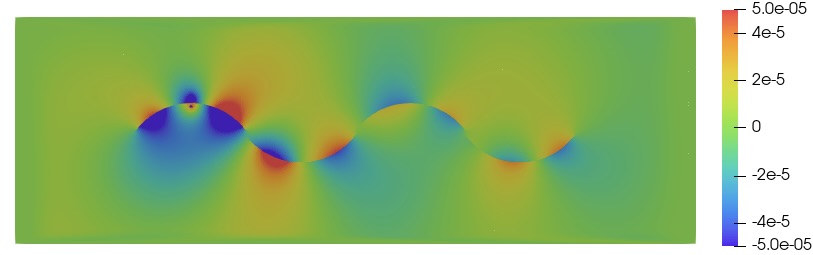}
  \end{minipage}
\hfill
 \end{center}
  \caption{Snapshots of the contour plots of $H_z$ at different time steps: 1000 (top left), 3600 (top right), 4400 (middle left), 6000 (middle right), 11000 (bottom left), and 18000 (bottom right). }
   \label{adjacent}
\end{figure}

\subsubsection{Example 3. A bulb shaped graphene sheet}
In this example, we present a numerical simulation of the SPP propagating along a bulb shaped graphene sheet. The setup with an exemplary coarse mesh is shown in Figure \ref{bulb_mesh}. One pair of dipole source wave is placed at $(-15 \mu \mathrm{m}, 2.5 \mu \mathrm{m}) \text{ and } (-15 \mu \mathrm{m}, 1.5 \mu \mathrm{m})$, and 
another pair is placed at $(-15 \mu \mathrm{m}, -2.5 \mu \mathrm{m}) \text{ and } (-15 \mu \mathrm{m}, -1.5 \mu \mathrm{m})$  with the same function given as Example 1. We adopt the same parameters as Example 1. 

The numerical results are shown in Figure \ref{Fig7}. As we can see, the surface wave propagates along the graphene sheet. 

\begin{figure}[H]
  \begin{center}
      \includegraphics[width=0.8\textwidth]{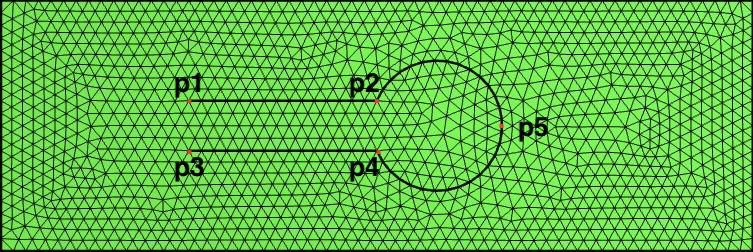}
 \caption{A bulb shaped graphene sheet buried in $\O$ with an exemplary coarse mesh. The structure is constructed by two segments $\mathrm{p}_1\mathrm{p}_2, \mathrm{p}_3\mathrm{p}_4$ and one circular arc. The coordinates for the marked points on the graph are correspondingly 
$\mathrm{p}_1(-15, 2)\mu \mathrm{m}, \mathrm{p}_2(0, 2)\mu \mathrm{m}, \mathrm{p}_3(-15, -2)\mu \mathrm{m}, \mathrm{p}_4(0, -2)\mu \mathrm{m}, \mathrm{p}_5(10, 0)\mu \mathrm{m}$.}
    \label{bulb_mesh}
  \end{center}
\end{figure}

\begin{figure}[H]
    \begin{minipage}[b]{\textwidth}
    \includegraphics[width=0.5\textwidth]{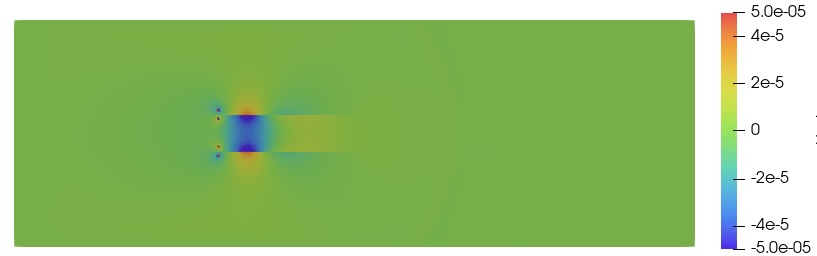}
    \includegraphics[width=0.5\textwidth]{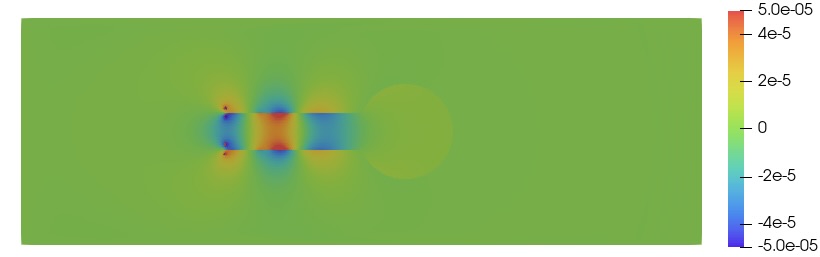}
    \includegraphics[width=0.5\textwidth]{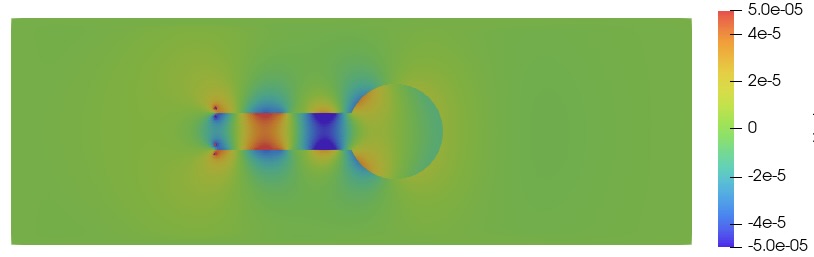}
    \includegraphics[width=0.5\textwidth]{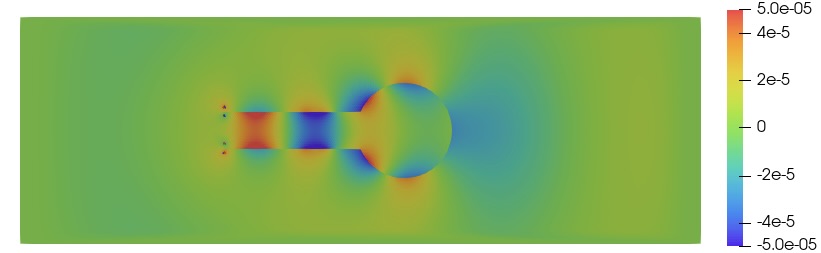}
    \includegraphics[width=0.5\textwidth]{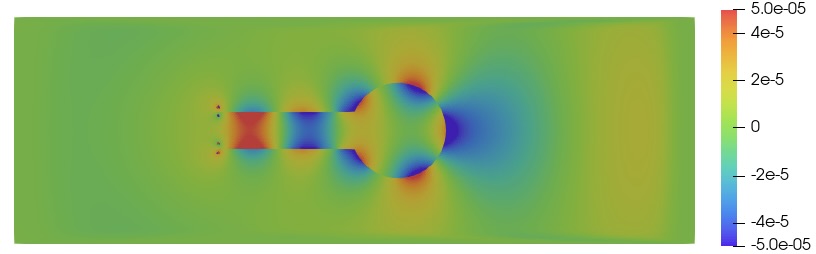}
    \includegraphics[width=0.5\textwidth]{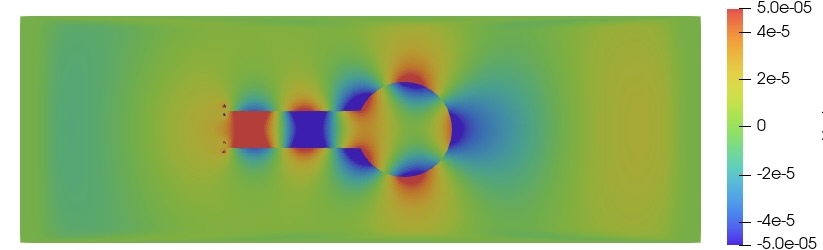}
  \end{minipage}
  \caption{Snapshots of the contour plots of $H_z$ obtained at different time steps: 100 (top left), 1800 (top right), 3000 (middle left), 4000 (middle right), 5200 (bottom left), and 10000 (bottom right). }
\label{Fig7}
\end{figure}

\subsubsection{Example 4: Ring resonator graphene interface}
In this example, we present a numerical simulation of an optical ring resonator motivated by the paper \cite{Davoodi2017}. The resonator consists of a graphene ring with a radius of $11\mu \mathrm{m}$, centered at the origin (0,0), embeded within the domain $\O_0$. Two graphene segments, each with $30\mu \mathrm{m}$ in length are positioned at $y=13\mu \mathrm{m}$ and $y=-13\mu \mathrm{m}$, respectively. The setup with a coarse mesh is shown in Figure 8. A pair of dipole source wave is placed at $(-13 \mu \mathrm{m}, 13.5\mu \mathrm{m}) \text{ and } (-13 \mu \mathrm{m}, 12.5 \mu \mathrm{m})$ with the same parameters as Example 1.

The numerical results are displayed in Figure \ref{ring}. The input wave propagates through the top graphene segment, reaching the transmission port and exiting to the left of the bottom graphene segment. This behavior mirrors the phenomenon observed in \cite[Fig.6(a)]{Davoodi2017}, where the input wave enters through the bottom graphene segment and exits to the left of the top graphene segment. The opposite direction in our setup is due to the position of the source wave. 

\begin{figure}[H]
\begin{center}
     \includegraphics[width=0.3\textwidth]{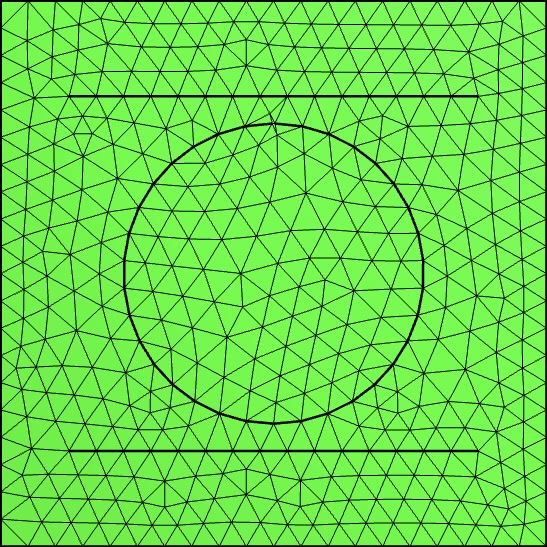}
\end{center}
 \caption{The graphene interface in $\O_2$ with a coarse mesh.}
\label{ring_mesh}
\end{figure}

\begin{figure}[H]
 \begin{center}
    \begin{minipage}[b]{\textwidth}
    \includegraphics[width=0.33\textwidth]{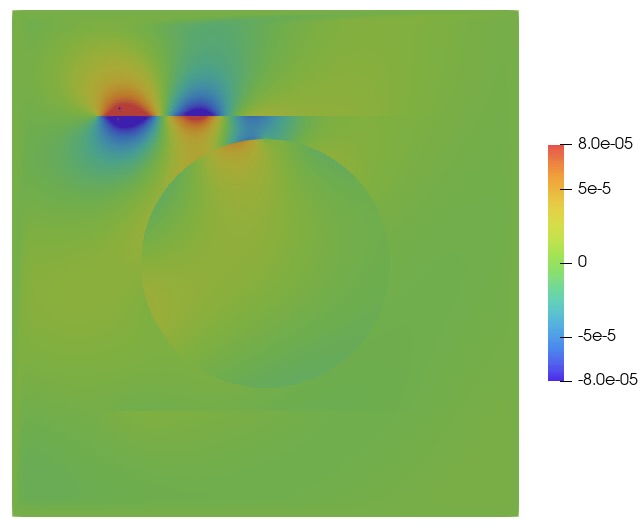}
    \includegraphics[width=0.33\textwidth]{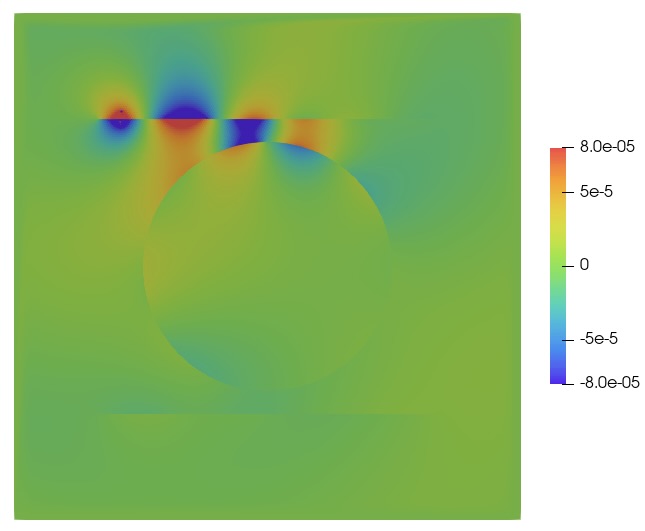}
    \includegraphics[width=0.33\textwidth]{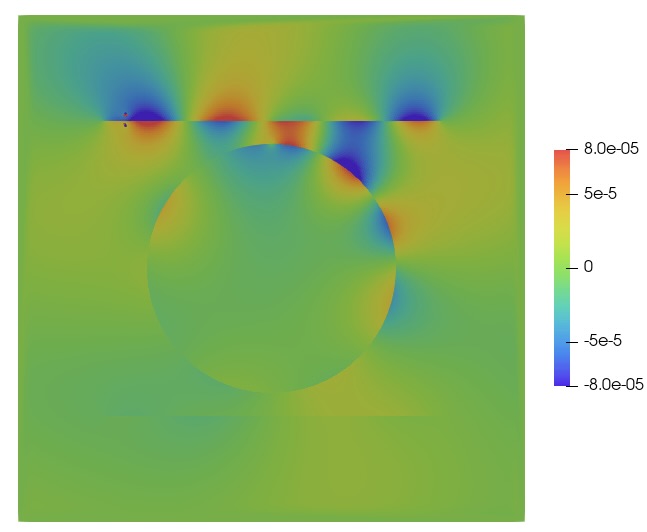}
    \end{minipage}

    \begin{minipage}[b]{\textwidth}
    \includegraphics[width=0.33\textwidth]{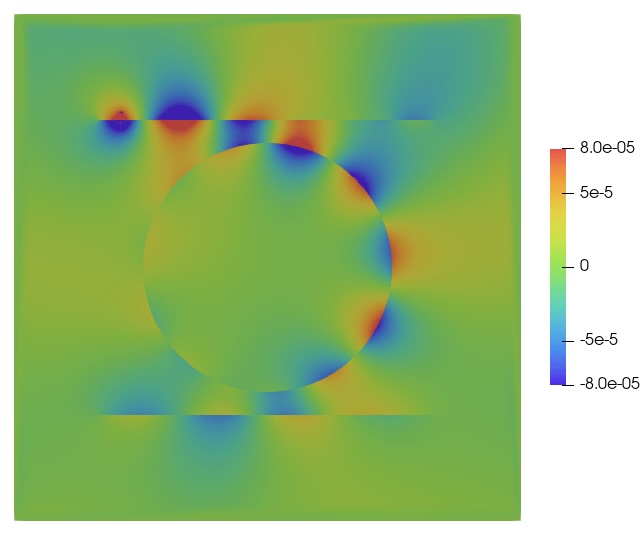}
    \includegraphics[width=0.33\textwidth]{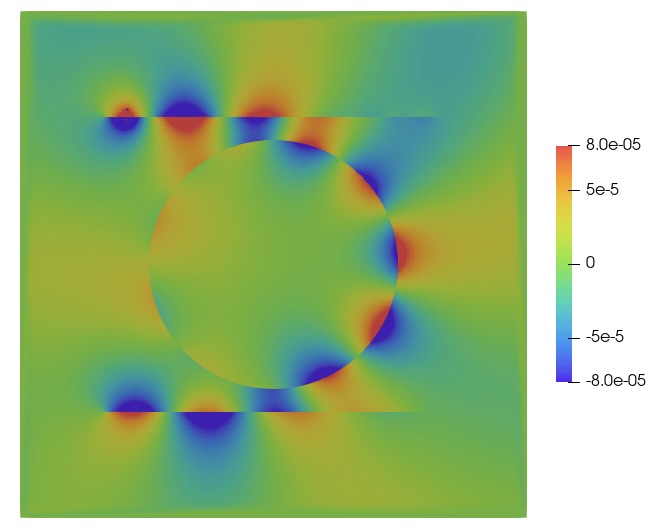}
    \includegraphics[width=0.33\textwidth]{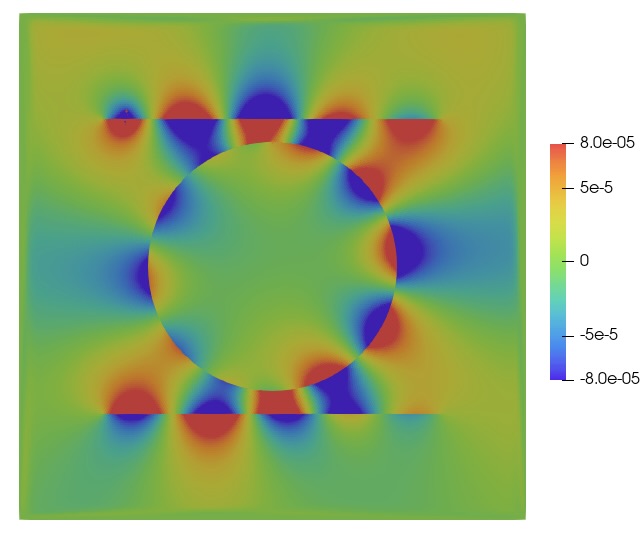}  \end{minipage}
 \end{center}
  \caption{Snapshots of the contour plots of $H_z$ obtained at different time steps:
2000 (top left), 3000 (top middle), 5000 (top right), 7500 (bottom left), 11250 (bottom middle), and 19000 (bottom right). }
  \label{ring}
\end{figure}

\subsubsection{Example 5: A spiral graphene interface}

In this example, we present a numerical simulation of SPP surface wave propagation along a spiral graphene interface inspired by \cite{wang2015broadband} with our scheme. A pair of dipole source wave is placed at $(-16 \mu \mathrm{m}, -18.5 \mu \mathrm{m}) \text{ and } (-16 \mu \mathrm{m}, -17.5 \mu \mathrm{m})$. 

The obtained numerical magnetic fields $H_z$ at various time steps are presented in Figure \ref{spiral}. As we can see, the wave propagates along the graphene interface. 
\begin{figure}[H]
\begin{center}
    \includegraphics[width=0.3\textwidth]{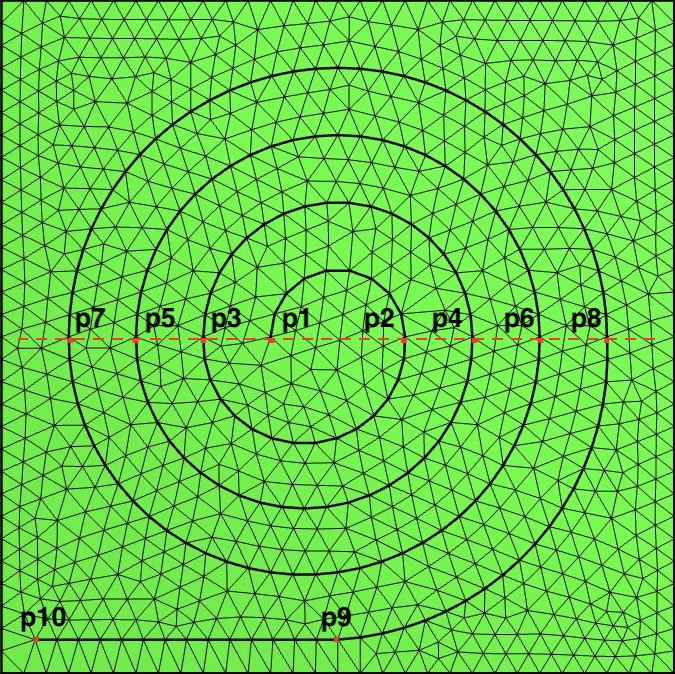}
 \caption{The graphene interface in $\O_2$ with a coarse mesh. The spiral structure is constructed by 7 semicircles, one quarter circle and one segment. Let $w=4 \mu \mathrm{m}$, and the coordinates for the marked points on the graph are correspondingly $p_1(-w,0), p_2(w,0), p_3(-2w,0), p_4(2w,0), p_5(-3w,0), p_6(3w,0), p_7(-4w,0), p_8(4w,0), p_9(0,-4.5w), \\p_{10}(-4.5w, -4.5w)$.} 
 \label{spiral_mesh}
\end{center}
\end{figure}
 
\begin{figure}[H]
 \begin{center}
    \begin{minipage}[b]{\textwidth}
    \includegraphics[width=0.33\textwidth]{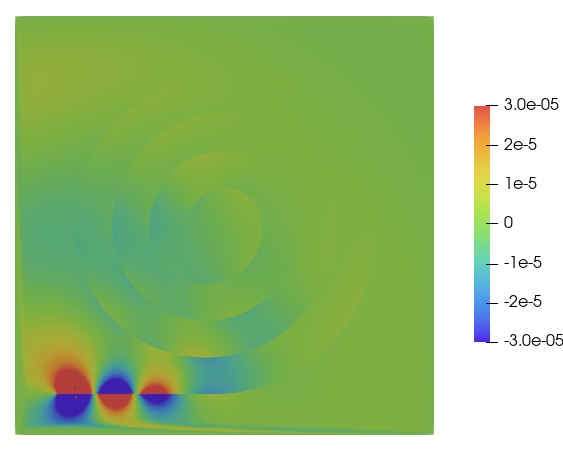}
    \includegraphics[width=0.33\textwidth]{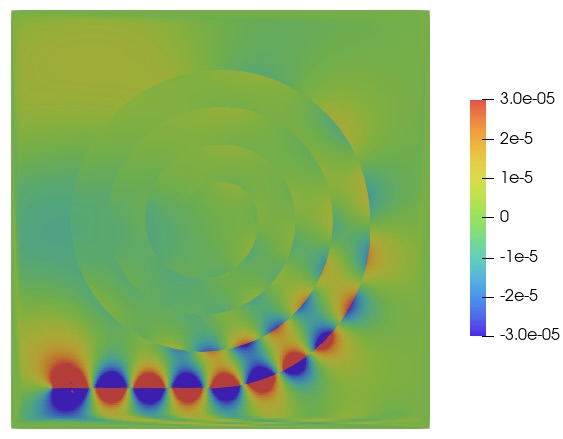}
    \includegraphics[width=0.33\textwidth]{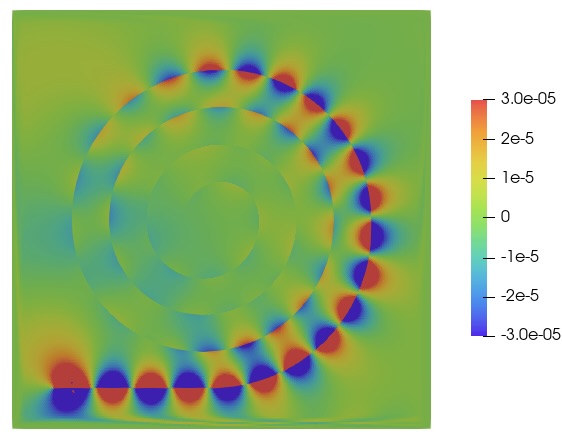}
    \includegraphics[width=0.33\textwidth]{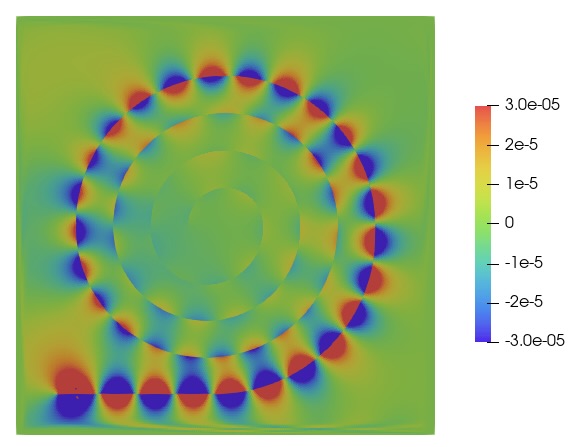}
    \includegraphics[width=0.33\textwidth]{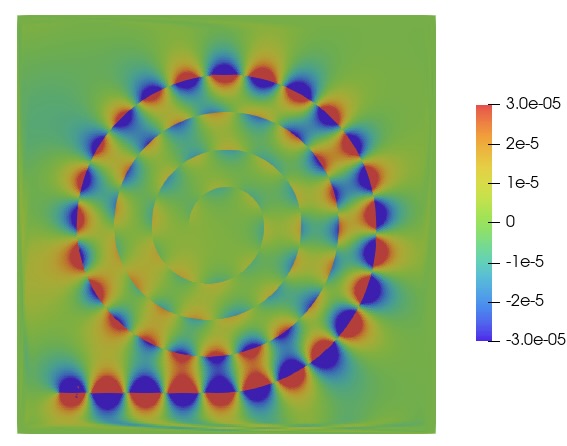}
    \includegraphics[width=0.33\textwidth]{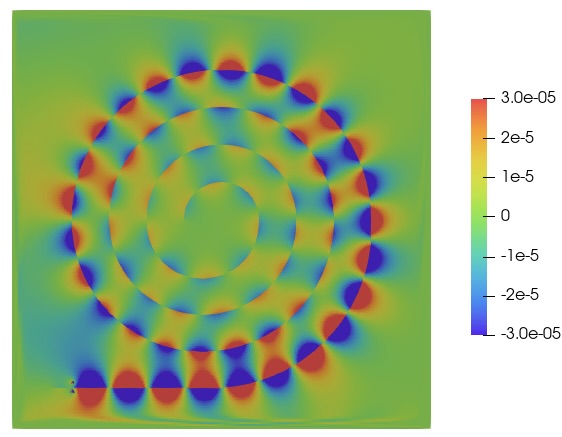}
  \end{minipage}
 \end{center}
  \caption{Snapshots of the contour plots of $H_z$ obtained at various time steps: 2000 (top left), 8000 (top middle), 20000 (top right), 32000 (bottom left), 72000 (bottom middle), and 100000 (bottom right). }
  \label{spiral}
\end{figure}

\section{Conclusion}

In this paper,  we first developed a simplified graphene model by eliminating the surface current variable from the graphene model adopted in our previous work \cite{Li_CMAWA2023}. Then we established the stability for the reformulated PDE model, and proposed a new finite element method for solving it. Extensive numerical simulations were carried out to demonstrate that the reformulated model captures the  surface plasmon polaritons very efficiently for various complex graphene sheets. Moreover, our new algorithm significantly improves computational efficiency, as shown by the time comparison results in Example 1, making it a more effective approach for simulating graphene-based plasmonic phenomena.

\vskip 0.1in
{\bf Acknowledgments.} 
Dr. J. Li is very grateful to Dr. Frederic Marazzato at University of Arizona for some insightful discussions. The authors like to thank the two anonymous referees for their insightful suggestions on improving our paper.

\section{Declarations}
{\bf Funding} Zhu's work is supported by the NSF grant 2136228 “RTG: Program in Computation- and Data-Enabled Science”.

\vskip 0.1in
{\bf Conflicts of Interests/Competing Interest} The authors have no conflicts of interest to declare that are relevant to the content of this article.

\end{document}